\documentclass[12pt,twoside,reqno]{amsart}
\usepackage{amsmath}
\usepackage{amsfonts}
\usepackage{amssymb}
\usepackage{color}
\usepackage{mathrsfs}
\usepackage{cite}
\DeclareMathAlphabet{\mathpzc}{OT1}{pzc}{m}{it}
\newtheorem{theorem}{Theorem}[section]

\newtheorem{lemma}[theorem]{Lemma}
\newtheorem{proposition}[theorem]{Proposition}

\newtheorem{remark}[theorem]{Remark}
\numberwithin{equation}{section}
\newcommand{\e}{\varepsilon}
\newcommand{\rd}{\mathrm{d}}
\begin{document}
\title{Vanishing aspect ratio limit for a fourth-order MEMS model} 
\thanks{Partially supported by the French-German PROCOPE  project 30718Z}

\author{Philippe Lauren\c{c}ot}
\address{Institut de Math\'ematiques de Toulouse, UMR~5219, Universit\'e de Toulouse, CNRS \\ F--31062 Toulouse Cedex 9, France}
\email{laurenco@math.univ-toulouse.fr}

\author{Christoph Walker}
\address{Leibniz Universit\"at Hannover, Institut f\"ur Angewandte Mathematik\\ Welfengarten 1, D -- 30167 Hannover, Germany}
\email{walker@ifam.uni-hannover.de}


\date{\today}

\begin{abstract}
So far most studies on mathematical models for microelectromechanical systems (MEMS) are focused on the so-called small aspect ratio model which is a wave or beam equation with a singular source term. It is formally derived by setting the aspect ratio equal to zero in a more complex model involving a moving boundary. A rigorous justification of this derivation is provided here when bending is taken into account.
\end{abstract}
%
\maketitle

%
%
\pagestyle{myheadings}
\markboth{\sc{Ph.~Lauren\c cot and Ch.~Walker}}{\sc{vanishing aspect ratio for a $4^\mathrm{th}$-order MEMS model}}

\section{Introduction}\label{sec1}

An idealized microelectromechanical system (MEMS) is made of a rigid  ground  plate located at height $z=-1$ and held at potential zero above which an elastic plate held at a constant potential normalized to one is suspended, the latter being at height $z=0$ at rest. Assuming that there is no variation in one of the horizontal directions, the electromechanical response of the device may be described by the elastic deformation $u=u(t,x)\ge -1$ of the elastic membrane and the electrostatic potential $\psi = \psi(t,x,z)$ between the two plates \cite[Section~7.4]{PB03}. The evolution of the electrostatic potential is given by the rescaled Poisson equation
\begin{equation}
\e^2 \partial_x^2 \psi(t,x,z) + \partial_z^2 \psi(t,x,z) = 0 \;\;\text{ in }\;\; \Omega(u(t))\ , \quad t>0\ , \label{eqpsi}
\end{equation}
in the domain 
$$
\Omega(u(t)) := \left\{ (x,z)\in I\times (-1,\infty)\ :\ -1 < z < u(t,x) \right\}\ , \qquad I:= (-1,1)\ ,
$$
between the plates, supplemented with the boundary conditions
\begin{equation}
\psi(t,x,z) = \frac{1+z}{1+u(t,x)} \;\;\text{ on }\;\; \partial\Omega(u(t))\ , \quad t>0\ , \label{bcpsi}
\end{equation}
while that of the deformation obeys a parabolic (if $\gamma=0$) or hyperbolic (if $\gamma>0$) equation
\begin{eqnarray}
& & \gamma^2 \partial_t^2 u(t,x) + \partial_t u(t,x) + \beta \partial_x^4 u(t,x) - \left( \tau + a \|\partial_x u(t)\|_2^2 \right) \partial_x^2 u(t,x) \nonumber\\
& & \qquad\qquad = - \lambda \left( \e^2 |\partial_x \psi(t,x,u(t,x))|^2 + |\partial_z \psi(t,x,u(t,x))|^2 \right)\ , \quad t>0\ , \quad x\in I\ , \label{equ}
\end{eqnarray}
supplemented with clamped boundary conditions
\begin{equation}
u(t,\pm 1 ) = \beta \partial_x u(t,\pm 1) = 0 \ , \qquad t>0\ , \label{bcu}
\end{equation}
and initial conditions
\begin{equation}
(u,\gamma \partial_t u)(0) = (u^0, \gamma u^1)\ , \qquad x\in I\ . \label{inu}
\end{equation} 
In \eqref{equ}, the terms corresponding to mechanical forces are $\beta\partial_x^4 u$, which accounts for plate bending, and $\left( \tau + a \|\partial_x u\|_2^2 \right) \partial_x^2 u$, which accounts for external stretching ($\tau>0$) and self-stretching due to moderately large oscillations ($a>0$). The right-hand side of \eqref{equ} describes the electrostatic forces exerted on the elastic plate which are proportional to the square of the trace of the (rescaled) gradient of the electrostatic potential on the elastic plate, the parameter $\lambda$ depending on the square of the applied voltage difference before scaling. Inertial forces are accounted for by $\gamma^2 \partial_t^2 u$ in \eqref{equ} with $\gamma\ge 0$ and $\partial_t u$ is a damping term which governs the dynamics in the damping dominated limit $\gamma=0$. An utmost important parameter in \eqref{eqpsi}-\eqref{inu} is the aspect ratio $\e>0$ which is proportional to the ratio height/length of the device. In applications this ratio is usually very small and therefore often taken to be zero what has far-reaching consequences. 

Indeed, setting $\e=0$ allows one to solve explicitly \eqref{eqpsi}-\eqref{bcpsi} in terms of the deformation $u_0$ as
$$
\psi_0(t,x,z) =  \frac{1+z}{1+u_0(t,x)}\,,\quad (x,z)\in \Omega(u_0 (t))\,,\quad t >0\,.
$$ 
Here, the subscript zero is used to indicate the formal limit $\e\to 0$. Furthermore, setting $\e=0$ in \eqref{equ} and using the previous formula for $\psi_0$, we are led to the so-called {\it small aspect ratio model}
\begin{equation}
\gamma^2 \partial_t^2 u_0 + \partial_t u_0 + \beta \partial_x^4 u_0 - \left( \tau + a \|\partial_x u_0\|_2^2 \right) \partial_x^2 u_0 = - \frac{\lambda}{(1+u_0)^2}\ , \quad t>0\ , \quad x\in I\ , \label{sg}
\end{equation}
supplemented with clamped boundary conditions \eqref{bcu} and initial conditions \eqref{inu}. This model is studied in several works in recent years,  among which we refer to \cite{COG09, CFT14, CEGM10, DFG10, EGG10, Flo14, Guo10, GLY14, KLNT11, KLNT15, LLZ14, LW14b, LiZ14, LiZ15, LiY07, LiL12, LLS13, LiW08,  LiW11}.

Take notice that the above computation is formal and a reliable use of the small aspect ratio model \eqref{sg} thus requires to justify it rigorously. As the computation involves a singular perturbation it is by no means obvious. The aim of this paper is to put this approximation on firm ground by providing a proof that solutions of \eqref{eqpsi}-\eqref{inu} indeed converge to that of \eqref{bcu}-\eqref{sg} as $\e\to 0$ when bending and self-stretching of the plate are taken into account, that is, when $\beta>0$ and $a\ge 0$. When these effects are not included, that is, when $\beta=a=0$, this issue  is  addressed in \cite{LW13} for the stationary problem and in \cite{ELW14} for the evolution problem in the damping dominated case $\gamma=0$. It  is  proved in these settings that solutions to \eqref{eqpsi}-\eqref{inu} converge as $\e\to 0$ to a solution to \eqref{bcu}-\eqref{sg} in suitable topologies. To achieve this result several difficulties are faced with some differences between the stationary and evolutionary settings. In fact, the main difficulty arising in the study of \eqref{eqpsi}-\eqref{inu} is the so-called \textit{pull-in instability} which manifests in the following ways. It corresponds to the non-existence of stationary solutions for values of $\lambda$ exceeding a threshold value depending on $\varepsilon$. For the evolution problem it corresponds to the occurrence of a finite time singularity  for values of $\lambda$ larger than a critical value depending not only on $\varepsilon$, but also on the initial condition $(u^0,\gamma u^1)$ in \eqref{inu}. Concerning the latter it means that the solution to \eqref{eqpsi}-\eqref{inu} only exists on a finite time interval $[0,T_{m}^\varepsilon)$ and satisfies 
$$
\lim_{t\to T_{m}^\varepsilon} \min_{x\in [-1,1]} u(t,x) = -1\ ,
$$
see \cite{LW14a}. To overcome these difficulties including the intricate dependence upon $\varepsilon$ and study the limiting behavior as $\varepsilon\to 0$, one has to find a range of values of $\lambda$ for which stationary solutions exist for all sufficiently small values of $\varepsilon$, respectively to derive a lower bound on the maximal existence time $T_m^\varepsilon$ which does not depend on $\varepsilon$. This approach  is  used in \cite{LW13, ELW14} when $\beta=\gamma=a=0$, but can only be adapted here to the evolution problem. Indeed, in contrast to \cite{LW13}, the construction of stationary solutions to \eqref{eqpsi}-\eqref{bcu} performed in \cite{LW_COCV} for $\beta>0$ and $a\ge 0$ does not provide an interval $(0,\Lambda)$ such that stationary solutions exist for all $\lambda \in (0,\Lambda)$ and $\varepsilon$ small enough. We shall thus develop an alternative argument in Section~\ref{sec3}. The second obstacle met in the analysis of the vanishing aspect ratio limit is the derivation of estimates on the solutions to \eqref{eqpsi}-\eqref{inu} which are uniform with respect to $\varepsilon$ small enough and sufficient to pass to the limit in \eqref{eqpsi}-\eqref{bcpsi} and the right-hand side of \eqref{equ}. According to the analysis performed in \cite{LW13, ELW14}, a bound on $u$ in $W_q^2(I)$ or $L_\infty(0,T,W_q^2(I))$ for some $q>2$ and $T>0$ is sufficient. For the evolution problem this bound is derived simultaneously with the aforementioned lower bound on the maximal existence time $T_m^\varepsilon$. For the stationary problem the situation is completely different as the first part of the analysis only provides a bound in $H^2(I)$. Proceeding in the study of the limit $\varepsilon\to 0$ requires a different approach to cope with weaker regularity on $u$ and is presented in Section~\ref{sec2}. As in \cite{LW_COCV} it is based on the observation that, since $\partial_x u$ vanishes on the boundary when $\beta>0$, the same regularity of the right-hand side of \eqref{equ} as in \cite{ELW14, LW14a} can be derived for less regular $u$. 

From now on we tacitly assume that
\begin{equation}
\beta>0\ , \quad \tau\ge 0 \ , \quad a\ge 0\ , \quad \gamma\ge 0 \ . \label{bach}
\end{equation} 
For further use, given $s>0$ and $q\in [1,\infty]$, we set
$$
W_{q,D}^s(I) := \left\{
\begin{array}{ll}
\left\{ v \in W_q^s(I)\ :\ v(\pm 1) = \beta\partial_x v(\pm 1)=0 \right\}\ , & \quad s>(q+1)/q\ , \\
\left\{ v \in W_q^s(I)\ :\ v(\pm 1) =0 \right\}\ , & \quad 1/q<s<(q+1)/q\ , \\
W_q^s(I) \ , & \quad s<1/q\ ,
\end{array}
\right. 
$$
and $H_D^s(I) := W_{2,D}^s(I)$. Also, for $\kappa \in (0,1)$, we define
$$
S_q^s(\kappa) := \left\{ v \in W_{q,D}^s(I)\ :\ v>-1+\kappa \;\text{ in }\; I \;\;\text{ and }\;\; \|v\|_{W_q^s} < \frac{1}{\kappa} \right\}\ .
$$ 

The first result deals with the vanishing aspect ratio limit for the stationary problem and the starting point is the existence result obtained in \cite{LW_COCV}. Let us first recall that \eqref{eqpsi}-\eqref{inu} has a variational structure and that stationary solutions are critical points of the total energy
\begin{equation}
\mathcal{E}_\e(v) := \mathcal{E}_m(v) - \lambda \mathcal{E}_{e,\e}(v)\ . \label{E}
\end{equation}
In \eqref{E}, the mechanical energy is defined by
\begin{equation}
\mathcal{E}_m(v) := \frac{\beta}{2} \|\partial_x^2 v\|_2^2 + \frac{1}{2} \left( \tau + \frac{a}{2} \|\partial_x v\|_2^2 \right) \|\partial_x v\|_2^2 \label{Em}
\end{equation}
and the electrostatic energy by
\begin{equation}
\mathcal{E}_{e,\e}(v) := \int_{\Omega(v)} \left( \varepsilon^2 |\partial_x \psi_v|^2 + |\partial_z \psi_v|^2 \right)\ \rd(x,z) \ , \label{Ee}
\end{equation}
where $\Omega(v) := \{ (x,z)\in I \times (-1,\infty)\ :\ -1 < z < v(x)\}$ and $\psi_v$ denotes the solution to the rescaled Poisson equation 
\begin{equation}
\varepsilon^2 \partial_x^2 \psi_v + \partial_z^2 \psi_v = 0 \;\;\text{ in }\;\; \Omega(v)\ , \qquad \psi_v(x,z) = \frac{1+z}{1+v(x)}\ , \quad (x,z)\in\partial\Omega(v)\ . \label{rP}
\end{equation}
Clearly, $\psi_v$ depends on $v$ in a nonlocal and nonlinear way. Observe that all the above quantities are well defined for $v\in H^2(I)$ satisfying $v>-1$ in $[-1,1]$. It is easy to see that, setting $\phi(x) := - (1-x^2)^2$ for $x\in [-1,1]$, there holds $\mathcal{E}_\e(\theta\phi)\to -\infty$ as $\theta\nearrow 1$. Consequently, $\mathcal{E}_\e$ is not bounded from below. Nevertheless, stationary solutions can be constructed by a constrained minimization approach. We recall the result obtained in \cite{LW_COCV}.

\begin{proposition}\label{prop1}
Let $\varepsilon>0$ and $\varrho\in (2,\infty)$. There are
\begin{itemize}
\item $u_{\varrho,\varepsilon}\in H_D^4(I)$ satisfying $-1 < u_{\varrho,\varepsilon}<0$ in $I$, 
\item $\psi_{\varrho,\varepsilon}\in H^2(\Omega(u_{\varrho,\varepsilon}))$ satisfying $\psi_{\varrho,\varepsilon} = \psi_{u_{\varrho,\varepsilon}}$, 
\item and $\lambda_{\varrho,\varepsilon}>0$ 
\end{itemize}
such that $(u_{\varrho,\varepsilon}, \psi_{\varrho,\varepsilon})$ is a stationary solution to \eqref{eqpsi}-\eqref{bcu} with $\lambda=\lambda_{\varrho,\varepsilon}$. Moreover, both $u_{\varrho,\varepsilon}$ and $\psi_{\varrho,\varepsilon}$ are even with respect to $x$.
In addition,
\begin{equation}
\mathcal{E}_m(u_{\varrho,\varepsilon}) = \min\{ \mathcal{E}_m(u)\ :\ u\in\mathcal{A}_{\varrho,\varepsilon} \} \label{s1a}
\end{equation}
where
\begin{equation}
\mathcal{A}_{\varrho,\varepsilon} := \left\{ v \in H^2_D(I)\ :\ -1<v\le 0 \;\text{ in }\; I , \ v \;\text{ is even and }\; \mathcal{E}_{e,\e}(v)=\varrho \right\}\ . \label{s1b}
\end{equation}
\end{proposition}

Given $\varepsilon>0$ another approach to construct (a smooth branch of) stationary solutions to \eqref{eqpsi}-\eqref{bcu} for small values of $\lambda$ is based on the Implicit Function Theorem \cite{LW14a}. These solutions actually coincide with the ones from Proposition~\ref{prop1} having an electrostatic energy $\mathcal{E}_{e,\e}(u_{\varrho,\varepsilon})=\varrho$ close to $2$. Note, however, that Proposition~\ref{prop1} provides additional stationary solutions with large electrostatic energies as one can show that  $\lambda_{\varrho,\varepsilon}\rightarrow 0$ as $\varrho\rightarrow \infty$, see \cite{LW_COCV}.

\medskip

We may now state the convergence result for stationary solutions and thus provide a rigorous justification of the stationary small aspect ratio limit.

\begin{theorem}\label{thm2}
Let $\varrho\in (2,\infty)$. There are a sequence $(\varepsilon_k)_{k\ge 1}$ with $\varepsilon_k\to 0$, $\kappa_\varrho\in (0,1)$, $u_\varrho\in S^4_2(\kappa_\varrho)$, and $\lambda_\varrho>0$ such that
\begin{equation*}
\lim_{k\to\infty} \left\{ |\lambda_{\varrho,\varepsilon_k} - \lambda_\varrho| + \| u_{\varrho,\varepsilon_k} - u_\varrho\|_{H^1} \right\} = 0\ ,
\end{equation*} 
where $u_\varrho$ is a solution to the (stationary) small aspect ratio model
\begin{equation}
\beta \partial_x^4 u_\varrho - \left( \tau + a \|\partial_x u_\varrho\|_2^2 \right) \partial_x^2 u_\varrho = - \frac{\lambda_\varrho}{(1+u_\varrho)^2} \;\;\text{ in }\;\; I\ , \qquad u_\varrho(\pm 1) = \beta \partial_x u_\varrho(\pm 1) = 0 \ , \label{vivaldi}
\end{equation}
satisfying
$$
\int_{-1}^1 \frac{\rd x}{1+u_\varrho(x)} = \varrho\ .
$$
In addition, $u_\varrho$ is even with
$$
\mathcal{E}_m(u_\varrho) = \min\{ \mathcal{E}_m(u)\ :\ u\in\mathcal{A}_{\varrho,0} \}\,,
$$
where
$$
\mathcal{A}_{\varrho,0} := \left\{ v \in H^2_D(I)\ :\ -1<v\le 0 \;\text{ in }\; I , \ v \;\text{ is even and }\; \int_{-1}^1 \frac{\rd x}{1+v(x)} =\varrho \right\}\ ,
$$
and 
$$
\lim_{k\to\infty} \int_{\Omega(u_{\varrho,\varepsilon_k})}\left\{ \left| \psi_{\varrho,\varepsilon_k}(x,z) - \frac{1+z}{1+u_\varrho(x)} \right|^2 + \left| \partial_z \psi_{\varrho,\varepsilon_k}(x,z) - \frac{1}{1+u_\varrho(x)} \right|^2 \right\}\ \rd(x,z) = 0\ .
$$
\end{theorem}

In fact, the convergence to zero of the sequence $(\partial_z \psi_{\varrho,\varepsilon_k} - 1/(1+u_\varrho))_k$ also holds true in $H^1(\Omega(u_{\varrho,\varepsilon_k}) )$.

A by-product of Theorem~\ref{thm2} is the existence of stationary solutions to the small aspect ratio model \eqref{vivaldi} which are minimizers to a constrained variational problem. While the latter seems to be new, the existence of solutions to \eqref{vivaldi} is derived in \cite{COG09, CEGM10, DFG10, EGG10, GLY14,LW14b,  LiY07, LiW08, LiW11} when $I$ is replaced by the unit ball of $\mathbb{R}^N$, $N\ge 1$.

\bigskip

We next turn to the hyperbolic evolution problem and report the following convergence result.

\begin{theorem}\label{Bq}
Let $\gamma>0$ and $\lambda>0$. Given $2\alpha\in (0,1/2)$ and $\kappa\in (0,1)$ let ${\bf u}^0 =(u^0,u^1)\in H_D^{4+2\alpha}(I)\times H_D^{2+2\alpha}(I)$ with $u^0 \in S_2^{2+2\alpha}(\kappa)$.
For $\varepsilon\in (0,1)$ let $(u_\varepsilon,\psi_\varepsilon)$ be the unique solution to \eqref{eqpsi}-\eqref{inu} on the maximal interval of existence $[0,T_m^\varepsilon)$ with $\psi_\varepsilon:=\psi_{u_\varepsilon}$. There are $T>0$ and $\kappa_0\in (0,1)$ such that $T_m^\varepsilon\ge T$ and $u_\varepsilon(t)\in S_2^{2+2\alpha}(\kappa_0)$ for all $(t,\varepsilon)\in [0,T]\times (0,1)$. Moreover, for any $\alpha'\in [0,\alpha)$, 
\begin{equation*}
\lim_{\varepsilon\to 0} \left\{ \sup_{t\in [0,T]} \|u_{\varepsilon}(t) - u_0(t)\|_{H^{2+2\alpha'}} + \sup_{t\in [0,T]} \|\partial_t u_\varepsilon(t) - \partial_t u_0(t) \|_{H^{2\alpha'}} \right\} = 0\ ,
\end{equation*}
where 
$$
u_0 \in W_1^2(0,T, H^{2\alpha'}(I)) \cap C^1([0,T],H^{2\alpha'}(I)) \cap C([0,T],H_D^{2+2\alpha'}(I)) \cap L_1(0,T, H_D^{4+2\alpha'}(I)) 
$$ 
is a strong solution to 
\begin{equation}
\gamma^2 \partial_t^2 u_0 + \partial_t u_0 + \beta \partial_x^4 u_0 - \left( \tau + a \|\partial_x u_0\|_2^2 \right)\, \partial_x^2 u_0 = - \frac{\lambda}{(1+u_0)^2} \;\;\text{ in } (0,T)\times I\ , \label{sg4.1}
\end{equation}
supplemented with clamped boundary conditions
\begin{equation}
u_0(t,\pm 1 ) = \beta \partial_x u_0(t,\pm 1) = 0 \ , \qquad t \in (0,T)\ , \label{sg4.2}
\end{equation}
and initial conditions
\begin{equation}
(u_0,\gamma \partial_t u_0)(0) = (u^0, \gamma u^1)\ , \qquad x\in I\ . \label{sg4.3}
\end{equation}
In addition,
\begin{equation}
\lim_{\varepsilon\to 0} \int_{\Omega(u_{\varepsilon}(t))}\left\{ \left| \psi_{\varepsilon}(t,x,z) - \frac{1+z}{1+u_0(t,x)} \right|^2 + \left| \partial_z \psi_{\varepsilon}(t,x,z) - \frac{1}{1+u_0(t,x)} \right|^2 \right\}\ \rd(x,z) = 0 \label{cvpsi}
\end{equation}
for all $t\in [0,T]$.
\end{theorem}

The minimal existence time $T>0$ depends only on the parameters from \eqref{bach} and on $\alpha$, $\|u^0\|_{H_D^{4+2\alpha}}$,  $\|u^1\|_{H_D^{2+2\alpha}}$, and $\kappa$. The convergence stated in Theorem~\ref{Bq} is unlikely to extend to arbitrary $T>0$ in general owing to the possibility of the occurrence of a finite time singularity as already mentioned. Still, Theorem~\ref{Bq} holds true for any $T>0$ if $\lambda$, $a$, $\|u^0\|_{H_D^{4+2\alpha}}$, and $\|u^1\|_{H_D^{2+2\alpha}}$ are sufficiently small. Let us also point out the convergence of the whole family $(u_\varepsilon)_{\varepsilon\in (0,1)}$ in contrast to the stationary problem, where the convergence only holds for a subsequence. This is due to the uniqueness of the solution $u_0$ to the limit problem \eqref{sg4.1}-\eqref{sg4.3}. 

The local well-posedness of \eqref{eqpsi}-\eqref{inu} is established in \cite{LW14a} for $a=0$, but the proof extends to $a>0$ as $v \mapsto \|\partial_x v\|_2^2 \partial_x^2 v$ is a locally Lipschitz continuous map from $H^{s+2}(I)$ to $H^{s}(I)$ for all $s>0$. Concerning the small aspect ratio equation \eqref{sg4.1}-\eqref{sg4.3}, most available results are actually devoted to the situation where both bending and self-stretching are neglected, that is, $\beta=a=0$, see \cite{Flo14, KLNT11, KLNT15, LLZ14, LiZ14, LiZ15}. As far as we know, the only contribution taking into account bending ($\beta>0$) with clamped boundary conditions \eqref{sg4.2} is \cite{LW14b} and the proof of local well-posedness performed there extends to the case $a>0$, that is, to \eqref{sg4.1}-\eqref{sg4.3}. The dynamics of \eqref{sg4.1} was also studied with other boundary conditions, namely, pinned boundary conditions $u(t,\pm 1)=\beta \partial_x^2 u(t,\pm 1) =0$ in \cite{CFT14, Guo10} and hinged or Steklov boundary conditions $u(t,\pm 1)=\beta \partial_x^2 u(t,\pm 1) -\beta d (1-\sigma)\partial_x u(t,\pm 1) = 0$, $d>0$, in \cite{CFT14}. 

The last result deals with the parabolic version of \eqref{eqpsi}-\eqref{inu} corresponding to the damping dominated limit $\gamma=0$ when inertia effects are neglected.
 
\begin{theorem}\label{thm3}
Let $\gamma=0$, $a=0$, and $\lambda>0$. Given $2\alpha\in (0,2)$ and $\kappa\in (0,1)$ let $u^0\in S_2^{2+2\alpha}(\kappa)$. For $\varepsilon\in (0,1)$ let $(u_\varepsilon,\psi_\varepsilon)$ be the unique solution to \eqref{eqpsi}-\eqref{inu} on the maximal interval of existence $[0,T_m^\varepsilon)$ with $\psi_\varepsilon:=\psi_{u_\varepsilon}$. There are $T>0$ and $\kappa_0\in (0,1)$ such that $T_m^\varepsilon\ge T$ and $u_\varepsilon(t)\in S_2^{2+2\alpha}(\kappa_0)$ for all $(t,\varepsilon)\in [0,T]\times (0,1)$. Moreover, for any $\alpha'\in [0,\alpha)$, 
\begin{equation*}
\lim_{\varepsilon\to 0} \sup_{t\in [0,T]} \|u_{\varepsilon}(t) - u_0(t)\|_{H^{2+2\alpha'}}  = 0\ , 
\end{equation*}
where 
$$
u_0 \in C^1([0,T],L_2(I)) \cap C([0,T],H_D^{2+2\alpha'}(I)) 
$$ 
is a strong solution to 
\begin{equation}
\partial_t u_0 + \beta \partial_x^4 u_0 - \tau \partial_x^2 u_0 = - \frac{\lambda}{(1+u_0)^2} \;\;\text{ in } (0,T)\times I\ , \label{psg4.1}
\end{equation}
supplemented with clamped boundary conditions
\begin{equation}
u_0(t,\pm 1 ) = \beta \partial_x u_0(t,\pm 1) = 0 \ , \qquad t\in (0,T)\ , \label{psg4.2}
\end{equation}
and initial conditions
\begin{equation}
u_0(0) = u^0\ , \qquad x\in I\ . \label{psg4.3}
\end{equation}
In addition, the convergence properties \eqref{cvpsi} are still valid for all $t\in [0,T]$.
\end{theorem}

As before, the minimal existence time $T>0$ depends only on the parameters from \eqref{bach} and on $\alpha$, $\|u^0\|_{H_D^{2+2\alpha}}$, and $\kappa$. It can be taken arbitrarily large provided $\lambda$ and $\|u^0\|_{H_D^{2+2\alpha}}$ are sufficiently small. The proof of Theorem~\ref{thm3} is similar to that of Theorem~\ref{Bq} and will thus be omitted. We just mention that the local well-posedness of \eqref{eqpsi}-\eqref{inu} for $\gamma=0$ and \eqref{psg4.1}-\eqref{psg4.3} are shown in \cite{LW14a} and \cite{LW14b}, respectively. Qualitative results on the behavior of solutions to \eqref{psg4.1}-\eqref{psg4.3} 
may be found in \cite{LiL12, LLS13}.

\section{The rescaled Poisson equation}\label{sec2}

This section is devoted to the study of the stability of solutions to the subproblem \eqref{eqpsi}-\eqref{bcpsi} as $\e\to 0$ when the function $u$ belongs to a suitable class. In fact, given a function $v\in S_q^s(\kappa)$ for some suitably chosen parameters $q>1$, $s\ge 1$, and $\kappa\in (0,1)$, we carefully study the behavior of the solution $\psi_v$ to \eqref{rP} as $\e\to 0$, paying special attention on the dependence upon $v$ with the aim of deriving eventually bounds that only depend on $q$, $s$, and $\kappa$. As already mentioned such an analysis has already been performed in \cite{ELW14, LW13} for $q>2$, $s=2$, and $\kappa\in (0,1)$, but it is not applicable for the proof of Theorem~\ref{thm2} for which we only have an $H^2$-estimate as a starting point. However, as we shall see below, if $v$ is sufficiently smooth with vanishing derivative on the edges of $I$, bounds on $\psi_v$ can be derived when $q=2$ and $s\in (3/2,2)$.  

\begin{proposition}\label{prop2}
Consider $s\in (3/2,2)$, $\nu\in (2-s,1/2)$, $\sigma\in [0,1/2)$, and $\kappa\in (0,1)$. There is a positive constant $C_{sg} = C_{sg}(s,\nu,\sigma,\kappa)$ such that, given $\e\in (0,1)$ and $v\in S_2^s(\kappa)$, the corresponding solution $\psi_v$ to \eqref{rP} satisfies 
\begin{align}
\|\psi_v - b_v\|_{L_2(\Omega(v))} + \left\| \partial_z\psi_v - \frac{1}{1+v} \right\|_{L_2(\Omega(v))} & \le C_{sg}\ \varepsilon\ , \label{errest1} \\
\left\| \partial_x \partial_z\psi_v + \frac{\partial_x v}{(1+v)^2} \right\|_{L_2(\Omega(v))} + \left\| g_\varepsilon(v) - \frac{1}{(1+v)^2} \right\|_{H^\sigma} & \le C_{sg}\ \varepsilon^{(1-2\nu)/(3-2\nu)}\ , \label{errest2} \\
\left\| \partial_z^2\psi_v \right\|_{L_2(\Omega(v))} & \le C_{sg}\ \varepsilon^{4(1-\nu)/(3-2\nu)}\ , \label{errest3}
\end{align}
where
\begin{equation*}
g_\varepsilon(v)(x) := \varepsilon^2 |\partial_x\psi_v(x,v(x))|^2 + |\partial_z \psi_v(x,v(x))|^2 \;\;\;\text{ and }\;\;\; b_v(x,z) := \frac{1+z}{1+v(x)}
\end{equation*}
for $(x,z)\in \Omega(v)$.
\end{proposition}

The remainder of this section is devoted to the proof of Proposition~\ref{prop2}. To give a hint on how we proceed we first perform some formal computations in Section~\ref{sec21}. A rigorous proof is then given in Section~\ref{sec22}. However, since some computations have already been done in \cite{LW_COCV} we shall not repeat them here but only state their outcome.

\subsection{A formal computation}\label{sec21}
   
Let $\kappa\in (0,1)$ and consider $v\in S_2^s(\kappa)$ for some $s>3/2$. We denote the corresponding solution to \eqref{rP} by $\psi = \psi_v$ and set
$$
\gamma_m(x) := \partial_z \psi(x,v(x))\ , \qquad x\in I\ ,
$$
assuming $\psi$ to be sufficiently smooth so that the above definition is meaningful. We multiply \eqref{rP} by $\partial_z^2\psi$ and integrate over $\Omega(v)$. This gives
$$
\e^2 \int_{\Omega(v)} \partial_x^2 \psi\ \partial_z^2 \psi\ \rd(x,z) + \int_{\Omega(v)} |\partial_z^2 \psi|^2\ \rd(x,z) =0\ .
$$
At least formally, we infer from Green's formula that
\begin{align*}
\int_{\Omega(v)} \partial_x^2 \psi\ \partial_z^2 \psi\ \rd(x,z) & = - \int_{\Omega(v)} \partial_x \psi\ \partial_z^2 \partial_x \psi\ \rd(x,z) - \int_{-1}^1 \left( \partial_x \psi\ \partial_z^2 \psi \right)(x,v(x)) \partial_x v(x)\ \rd x \\
& \qquad + \int_{-1}^0 \left[ \left( \partial_x \psi\ \partial_z^2 \psi \right)(1,z) - \left( \partial_x \psi\ \partial_z^2 \psi \right)(-1,z) \right]\ \rd z\,.
\end{align*}
Owing to the boundary condition $\psi(\pm 1,z)=1+z$ we deduce (at least formally) that $\partial_z^2 \psi(\pm 1,z)=0$ for $z\in (-1,0)$ and thus the last term of the right-hand side of the previous identity vanishes. We then use once more Green's formula and find
\begin{align*}
\int_{\Omega(v)} \partial_x^2 \psi\ \partial_z^2 \psi\ \rd(x,z) & = \int_{\Omega(v)} |\partial_x \partial_z \psi|^2\ \rd(x,z) + \int_{-1}^1 \left( \partial_x \psi\ \partial_x \partial_z \psi \right)(x,-1)\ \rd x \\
& \qquad - \int_{-1}^1 \left( \partial_x \psi\ \partial_x \partial_z \psi \right)(x,v(x))\ \rd x - \int_{-1}^1 \left( \partial_x \psi\ \partial_z^2 \psi \right)(x,v(x)) \partial_x v(x)\ \rd x
\end{align*}
Now we note that $\psi(x,-1)=0$ for $x\in I$ so that $\partial_x \psi(x,-1)=0$ and that
$$
\partial_x \gamma_m(x) = \partial_x \partial_z \psi (x,v(x)) + \partial_x v(x)\ \partial_z^2 \psi(x,v(x)) \ , \qquad x\in I\ , 
$$
again formally. Therefore the above identity reduces to
$$
\int_{\Omega(v)} \partial_x^2 \psi\ \partial_z^2 \psi\ \rd(x,z) = \int_{\Omega(v)} |\partial_x \partial_z \psi|^2\ \rd(x,z) - \int_{-1}^1 \partial_x \psi(x,v(x)) \partial_x \gamma_m(x)\ \rd x \ .
$$
Finally, differentiating the boundary condition $\psi(x,v(x))=1$ with respect to $x\in I$ gives 
$$
\partial_x \psi(x,v(x)) = - \gamma_m(x) \partial_x v(x)\ , \qquad x\in I\ ,
$$ 
and we end up with
$$
\int_{\Omega(v)} \partial_x^2 \psi\ \partial_z^2 \psi\ \rd(x,z) = \int_{\Omega(v)} |\partial_x \partial_z \psi|^2\ \rd(x,z) + \int_{-1}^1 \gamma_m(x) \partial_x \gamma_m(x) \partial_x v(x)\  \rd x \ .
$$
Summarizing we have
$$
\int_{\Omega(v)} \left[ \e^2 |\partial_x\partial_z \psi|^2 + |\partial_z^2\psi|^2 \right]\ \rd(x,z) = -\e^2 \int_{-1}^1 \gamma_m(x) \partial_x \gamma_m(x) \partial_x v(x)\  \rd x \ .
$$
Integrating the right-hand side of the above inequality by parts and using $\partial_x v(\pm 1)=0$ we conclude that
\begin{equation}
\int_{\Omega(v)} \left[ \e^2 |\partial_x\partial_z \psi|^2 + |\partial_z^2\psi|^2 \right]\ \rd(x,z) = \frac{\e^2}{2} \int_{-1}^1 \gamma_m(x)^2 \partial_x^2 v(x)\  \rd x\,. \label{f1}
\end{equation}
Note that the left-hand side of the above identity is equivalent to the $L_2(\Omega(v))$-norm of the gradient of $\partial_z \psi$. The key observation is now that the continuity of pointwise multiplication (see \cite[Theorem~4.1 \& Remark~4.2(d)]{Am91})
$$
H^{1/2}(I) \cdot H^{\nu}(I) \longrightarrow H^{2-s}(I)\ , \qquad 2-s<\nu < 1/2\ , 
$$
and the properties of the trace operator guarantee that $\gamma_m^2\in H^{2-s}(I)$ with  
\begin{align}
\left\| \gamma_m^2 \right\|_{H^{2-s}} & \le C\, \|\gamma_m\|_{H^{1/2}} \|\gamma_m\|_{H^\nu} \le C(\Omega(v))\, \|\partial_z\psi\|_{H^1(\Omega(v))} \|\partial_z\psi\|_{H^{(1+2\nu)/2}(\Omega(v))} \nonumber\\
& \le C(\Omega(v))\, \|\partial_z\psi\|_{H^1(\Omega(v))}^{(3+2\nu)/2} \|\partial_z\psi\|_{L_2(\Omega(v))}^{(1-2\nu)/2} \,.\label{f2}
\end{align}
Combining \eqref{f1} and \eqref{f2} we are led to 
\begin{align*}
\min\{\e^2, 1\} \|\partial_z\psi\|_{H^1(\Omega(v))}^2 & \le C\, \e^2 \left\| \gamma_m^2 \right\|_{H^{2-s}} \|\partial_x^2 v\|_{H^{s-2}}  + \varepsilon^2 \|\partial_z\psi\|_{L_2(\Omega(v))}^2 \\
& \le C(\Omega(v)) \e^2 \|\partial_z\psi\|_{H^1(\Omega(v))}^{(3+2\nu)/2} \|\partial_z\psi\|_{L_2(\Omega(v))}^{(1-2\nu)/2} \|v\|_{H^{s}}  + \varepsilon^2 \|\partial_z\psi\|_{L_2(\Omega(v))}^2 \ ,
\end{align*}
hence
$$
\min\{\e^2, 1\} \|\partial_z\psi\|_{H^1(\Omega(v))}^{(1-2\nu)/2} \le C(\Omega(v),\kappa) \e^2 \|\partial_z\psi\|_{L_2(\Omega(v))}^{(1-2\nu)/2}\ .
$$
We have thus shown that the $H^1(\Omega(v))$-norm of $\partial_z \psi$ can be controlled by its $L_2(\Omega(v))$-norm. It remains to justify rigorously the above computations and show that the constant stemming from the continuity of the trace operator which depends on $\Omega(v)$ actually only depends on $\kappa$. This is the purpose of the next section.

\subsection{Proof of Proposition~\ref{prop2}}\label{sec22}

Let $s\in (3/2,2]$, $\kappa\in (0,1)$, and $v\in S_2^s(\kappa)$. We denote the corresponding solution to \eqref{rP} by $\psi=\psi_v$ and recall that $\psi\in H^s(\Omega(v))$ by \cite[Corollary~4.2]{LW_COCV}. As in previous works, to cope with the dependence of the domain on $v$ we map the domain $\Omega(v)$ onto the rectangle $\Omega := (-1,1)\times (0,1)$ with the help of the transformation 
$$
T_v(x,z) := \left( x , \frac{1+z}{1+v(x)} \right)\ , \qquad (x,v)\in \Omega(v)\ .
$$ 
We then define 
\begin{equation}
\Phi(x,\eta) = \Phi_v(x,\eta) := \psi \circ T_v^{-1}(x,\eta) - \eta = \psi(x,-1+\eta(1+v(x))) - \eta\ , \qquad (x,\eta)\in \Omega\ , \label{c0}
\end{equation}
and recall that $\Phi\in H^{\sigma}(\Omega)$ for each $\sigma<s$ and $\partial_\eta\Phi\in H^1(\Omega)$ according to \cite[Proposition~4.1]{LW_COCV}. As a consequence of \eqref{rP} we realize that $\Phi$ solves the Dirichlet problem
\begin{equation}
\mathcal{L}_v \Phi = f_v \;\;\text{ in }\;\;\Omega\ , \qquad \Phi=0 \;\;\text{ on }\;\; \partial\Omega\ , \label{c1}
\end{equation}
the operator $\mathcal{L}_v$ and the function $f_v$ being given by
\begin{equation}
\begin{split}
\mathcal{L}_v w & := \varepsilon^2 \partial_x^2 w - 2\varepsilon^2\eta \frac{\partial_x v(x)}{1+v(x)}\ \partial_x \partial_\eta w + \frac{1+\varepsilon^2\eta^2 (\partial_x v(x))^2}{(1+v(x))^2} \partial_\eta^2 w \\
& \qquad + \varepsilon^2 \eta \left[ 2 \left( \frac{\partial_x v(x)}{1+v(x)} \right)^2 - \frac{\partial_x^2 v(x)}{1+v(x)} \right] \partial_\eta w
\end{split}\label{c2}
\end{equation}
and 
\begin{equation}
f_v(x,\eta) := \varepsilon^2 \eta \left[ \frac{\partial_x^2 v(x)}{1+v(x)} - 2 \left( \frac{\partial_x v(x)}{1+v(x)} \right)^2 \right]\ . \label{c3}
\end{equation}

We now report some useful identities involving $\Phi$.

\begin{lemma}\label{lem1}
We set $V := \partial_x (1+\ln{v}) = \partial_x v/(1+v) \in H^{s-1}(I)$ and $\Gamma := \partial_\eta\Phi(.,1) \in H^{1/2}(I)$.  Then $\Gamma^2\in H^{2-s}(I)$ and there hold
\begin{align}
\e^2 \left\| \partial_x \Phi - \eta V \partial_\eta\Phi \right\|_{L_2(\Omega)}^2 + \left\| \frac{\partial_\eta \Phi}{1+v} \right\|_{L_2(\Omega)}^2 & = \e^2 \int_\Omega V (\Phi+\eta) \left( \partial_x \Phi - \eta V \partial_\eta\Phi \right)\ \rd(x,\eta) \nonumber\\
& \qquad - \e^2 \int_\Omega \eta V^2 \Phi\ \rd(x,\eta)\ , \label{cd}
\end{align}
and
\begin{equation}
Q^2 := \e^2 \left\| \partial_x\partial_\eta \Phi - \eta V \partial_\eta^2 \Phi \right\|_{L_2(\Omega)}^2 + \left\| \frac{\partial_\eta^2 \Phi}{1+v} \right\|_{L_2(\Omega)}^2 = \e^2 \left( \mathcal{R}_1^* +\mathcal{R}_2^* \right)\ , \label{ce}
\end{equation}
with
\begin{align}
\mathcal{R}_1^* & := \frac{1}{2}\, \langle \partial_x V , \Gamma^2  \rangle_{H^{s-2},H^{2-s}} + \int_\Omega V \left( \partial_x\partial_\eta\Phi - \eta V \partial_\eta^2\Phi \right) \partial_\eta\Phi\ \rd(x,\eta)\ , \label{cf} \\
\mathcal{R}_2^* & := \langle \partial_x V , \Gamma  \rangle_{H^{s-2},H^{2-s}} -  \int_\Omega \eta V^2 \partial_\eta^2\Phi\ \rd(x,\eta)\ . \label{cg} 
\end{align}
\end{lemma}

\begin{proof}
It first follows from the continuity of pointwise multiplication (see \cite[Theorem~4.1 \& Remark~4.2(d)]{Am91})
$$
H^{1/2}(I) \cdot H^{1/2}(I) \longrightarrow H^{\sigma}(I)\ , \qquad 0 \le \sigma < 1/2\ ,
$$
that $\Gamma^2$ belongs to $H^\sigma(I)$ for all $\sigma \in [0,1/2)$ and in particular to $H^{2-s}(I)$. 

To prove \eqref{cd} and \eqref{ce} we first handle the case where $v$ is more regular, namely we assume that $v\in H^2(I)$ in addition to being in $S_2^s(\kappa)$. We then multiply \eqref{c1} by $\Phi$ and integrate over $\Omega$. Using Green's formula and the boundary condition $\Phi=0$ on $\partial\Omega$, we obtain
$$
\e^2 \left\| \partial_x \Phi - \eta V \partial_\eta \Phi \right\|_{L_2(\Omega)}^2 + \left\| \frac{\partial_\eta \Phi}{1+v} \right\|_{L_2(\Omega)}^2 = \e^2 \int_\Omega \eta (V^2 - \partial_x V) \Phi \left( 1 - \partial_\eta\Phi \right)\ \rd(x,\eta)\ .
$$
Moreover,
$$
- \int_\Omega \eta \Phi \partial_x V\ \rd(x,\eta) = \int_\Omega \eta V \partial_x \Phi\ \rd(x,\eta)
$$
and, employing Green's formula twice,
$$
\int_\Omega \eta \Phi \partial_\eta\Phi\ \partial_x V\ \rd(x,\eta) = -\frac{1}{2} \int_\Omega \Phi^2 \partial_x V\ \rd(x,\eta)  = \int_\Omega V \Phi \partial_x \Phi\ \rd(x,\eta)\ .
$$
Combining the above identities leads us to
\begin{align*}
\e^2 \left\| \partial_x \Phi - \eta V \partial_\eta \Phi \right\|_{L_2(\Omega)}^2 + \left\| \frac{\partial_\eta \Phi}{1+v} \right\|_{L_2(\Omega)}^2 & = \e^2 \int_\Omega V \Phi \left( \partial_x \Phi - \eta V \partial_\eta \Phi \right)\ \rd(x,\eta) \\
& \qquad + \int_\Omega \eta V \left( \partial_x \Phi + V \Phi \right)\ \rd(x,\eta) \ .
\end{align*}
We finally use once more the homogeneous Dirichlet boundary conditions of $\Phi$ and Green's formula to prove that
\begin{align*}
\int_\Omega \eta V^2 \Phi\ \rd(x,\eta) & = 2 \int_\Omega \eta V^2 \Phi\ \rd(x,\eta) - \int_\Omega \eta V^2 \Phi\ \rd(x,\eta) \\ 
& = - \int_\Omega \eta^2 V^2 \partial_\eta \Phi\ \rd(x,\eta) -\int_\Omega \eta V^2 \Phi\ \rd(x,\eta)
\end{align*}
and complete the proof of \eqref{cd}.

We next infer from \cite[Eqs.~(4.14)-(4.16)]{LW_COCV} that 
$$
Q^2 = \e^2 \left\| \partial_x\partial_\eta \Phi - \eta V \partial_\eta^2 \Phi \right\|_{L_2(\Omega)}^2 + \left\| \frac{\partial_\eta^2 \Phi}{1+v} \right\|_{L_2(\Omega)}^2 = \e^2 \left( \mathcal{R}_1 +\mathcal{R}_2 \right)\ , 
$$
with
\begin{align*}
\mathcal{R}_1 & := \int_\Omega \eta \left( \partial_x V - V^2 \right) \partial_\eta\Phi\ \partial_\eta^2\Phi\ \rd(x,\eta)\ , \\
\mathcal{R}_2 & := \int_\Omega \eta \left( \partial_x V - V^2 \right) \partial_\eta^2\Phi\ \rd(x,\eta)\ . 
\end{align*}
We then argue as in the derivation of \cite[Eq.~(4.17)]{LW_COCV} to show that $\mathcal{R}_1$ coincides with $\mathcal{R}_1^*$. Finally, note that
\begin{align*}
\mathcal{R}_2 & = \int_\Omega \partial_x V\ \partial_\eta \left( \eta \partial_\eta \Phi \right)\ \rd(x,\eta) - \int_\Omega \partial_x V \partial_\eta\Phi\ \rd(x,\eta) - \int_\Omega \eta V^2 \partial_\eta^2\Phi\ \rd(x,\eta) \\
& = \int_{-1}^1 \partial_x V \Big[ \eta \partial_\eta \Phi \Big]_{\eta=0}^{\eta=1}\ \rd x - \int_{-1}^1 \partial_x V \left[ \Phi(x,1)-\Phi(x,0) \right]\ \rd x  - \int_\Omega \eta V^2 \partial_\eta^2\Phi\ \rd(x,\eta) \\
& = \int_{-1}^1 \Gamma \partial_x V\ \rd x  - \int_\Omega \eta V^2 \partial_\eta^2\Phi\ \rd(x,\eta)\ ,
\end{align*}
so that $\mathcal{R}_2 = \mathcal{R}_2^*$ as claimed. 

Finally, extending the validity of \eqref{cd} and \eqref{ce} to functions $v$ belonging only to $S_2^s(\kappa)$ is done by an approximation argument as in \cite[Section~4]{LW_COCV} to which we refer.
\end{proof}

\begin{remark}
The boundary conditions $\partial_x v(\pm1)=0$ are used in the derivation of the identity $\mathcal{R}_1^*=\mathcal{R}_1$, see \cite{LW_COCV}.
\end{remark}

\begin{proof}[Proof of Proposition~\ref{prop2}]
On the one hand, one easily checks that the functions $\eta\mapsto -\eta$ and $\eta\mapsto 1-\eta$ solve \eqref{c1} in $\Omega$ with $-\eta \le \Phi(x,\eta) = 0 \le 1-\eta$ for all $(x,\eta)\in \partial\Omega$. We then deduce from the comparison principle that
\begin{equation}
- \eta \le \Phi(x,\eta) \le 1 - \eta\ , \qquad (x,\eta)\in\Omega\ . \label{x1}
\end{equation}
On the other hand, $H^s(I)$ is continuously embedded in $C^1([-1,1])$ and $H^{s-1}(I)$ is an algebra since $s>3/2$. Thus, as $v\in S_2^s(\kappa)$, there is $c_1(\kappa)>0$ depending only on $s$ and $\kappa$ such that
\begin{equation}
\|v\|_{C^1([-1,1])} + \| V \|_{\infty} + \|V\|_{H^{s-1}} \le c_1(\kappa)\ . \label{x2}
\end{equation}
It now follows from \eqref{cd}, \eqref{x1}, \eqref{x2}, and the Cauchy-Schwarz inequality that
\begin{align*}
& \e^2 \left\| \partial_x \Phi - \eta V \partial_\eta\Phi \right\|_{L_2(\Omega)}^2 + \left\| \frac{\partial_\eta \Phi}{1+v} \right\|_{L_2(\Omega)}^2 \\ 
& \qquad \le \frac{\e^2}{2} \int_\Omega \left( \partial_x \Phi - \eta V \partial_\eta\Phi \right)^2 \ \rd(x,\eta) + \frac{\e^2}{2} \int_\Omega V^2 \left[ (\Phi+\eta)^2 - 2 \eta \Phi \right]\ \rd(x,\eta) \\ 
& \qquad \le \frac{\e^2}{2} \left\| \partial_x \Phi - \eta V \partial_\eta\Phi \right\|_{L_2(\Omega)}^2 + \e^2 \left\| V \right\|_2^2 \\
& \qquad \le \frac{\e^2}{2} \left\| \partial_x \Phi - \eta V \partial_\eta\Phi \right\|_{L_2(\Omega)}^2 + c_1(\kappa)^2 \e^2\ .
\end{align*}
We have thus shown that
\begin{equation}
\e^2 \left\| \partial_x \Phi - \eta V \partial_\eta\Phi \right\|_{L_2(\Omega)}^2 + \left\| \frac{\partial_\eta \Phi}{1+v} \right\|_{L_2(\Omega)}^2 \le c(\kappa) \e^2\ . \label{ch}
\end{equation}
We next estimate $\mathcal{R}_1^*$. Fix $\nu\in (2-s,1/2)$. Then continuity of pointwise multiplication (see \cite[Theorem~4.1 \& Remark~4.2(d)]{Am91})
$$
H^{1/2}(I) \cdot H^\nu(I) \longrightarrow H^{2-s}(I)
$$
and \eqref{x2} imply that
$$
\left| \langle \partial_x V , \Gamma^2 \rangle_{H^{s-2},H^{2-s}} \right| \le \|\partial_x V \|_{H^{s-2}} \|\Gamma^2\|_{H^{2-s}} \le c \|V\|_{H^{s-1}}\ \|\Gamma\|_{H^{1/2}} \|\Gamma\|_{H^\nu} \le c(\kappa) \|\Gamma\|_{H^{1/2}} \|\Gamma\|_{H^\nu}\ .
$$
We next use the continuity of the trace from $H^\sigma(\Omega)$ in $H^{\sigma-1/2}(\partial\Omega)$ for all $\sigma\in (1/2,1]$ and the fact that the complex interpolation space $[L_2(\Omega) , H^1(\Omega)]_{\sigma}$ coincides with $H^{\sigma}(\Omega)$ (up to equivalent norms) to estimate $\Gamma$ and obtain
\begin{align}
\|\Gamma\|_{H^{1/2}} \le c \|\partial_\eta \Phi\|_{H^1(\Omega)} \;\;\text{ and }\;\; \|\Gamma\|_{H^{\sigma-1/2}} \le c \|\partial_\eta \Phi\|_{H^{\sigma}(\Omega)} \le c \|\partial_\eta \Phi\|_{H^1(\Omega)}^{\sigma} \|\partial_\eta \Phi\|_{L_2(\Omega)}^{1-\sigma}\ . \label{x2.5}
\end{align}
Combining the above estimate with $\eqref{x2.5}$ for $\sigma=(2\nu+1)/2$ gives
$$
\left| \langle \partial_x V , \Gamma^2 \rangle_{H^{s-2},H^{2-s}} \right| \le c(\kappa) \|\partial_\eta \Phi\|_{H^1(\Omega)}^{(2\nu+3)/2} \|\partial_\eta \Phi\|_{L_2(\Omega)}^{(1-2\nu)/2}\ .
$$
Since $\e\in (0,1)$ and $v\in S_2^s(\kappa)$ we further infer from \eqref{x2}, \eqref{ch}, and the definition \eqref{ce} of $Q$ that
\begin{align*}
\left| \langle \partial_x V , \Gamma^2 \rangle_{H^{s-2},H^{2-s}} \right| & \le c(\kappa) \left( \|\partial_x \partial_\eta \Phi\|_{L_2(\Omega)}^{(2\nu+3)/2} + \|\partial_\eta^2 \Phi\|_{L_2(\Omega)}^{(2\nu+3)/2} \right) \|\partial_\eta \Phi\|_{L_2(\Omega)}^{(1-2\nu)/2} + c(\kappa) \|\partial_\eta \Phi\|_{L_2(\Omega)}^2 \\
& \le c(\kappa) \|\partial_x \partial_\eta \Phi - \eta V \partial_\eta^2\Phi \|_{L_2(\Omega)}^{(2\nu+3)/2} \left\| \frac{\partial_\eta \Phi}{1+v} \right\|_{L_2(\Omega)}^{(1-2\nu)/2} \\
& \quad + c(\kappa) \left( \| \eta V \partial_\eta^2\Phi \|_{L_2(\Omega)}^{(2\nu+3)/2} + \|\partial_\eta^2 \Phi\|_{L_2(\Omega)}^{(2\nu+3)/2} \right) \left\| \frac{\partial_\eta \Phi}{1+v} \right\|_{L_2(\Omega)}^{(1-2\nu)/2} \\
& \quad + c(\kappa) \left\| \frac{\partial_\eta \Phi}{1+v} \right\|_{L_2(\Omega)}^2 \\
& \le c(\kappa) \left( \|\partial_x \partial_\eta \Phi - \eta V \partial_\eta^2\Phi \|_{L_2(\Omega)}^{(2\nu+3)/2} + \left\| \frac{\partial_\eta^2 \Phi}{1+v} \right\|_{L_2(\Omega)}^{(2\nu+3)/2} \right) \varepsilon^{(1-2\nu)/2} + c(\kappa) \varepsilon^2 \\
& \le c(\kappa) \left( \frac{Q}{\varepsilon} \right)^{(2\nu+3)/2} \varepsilon^{(1-2\nu)/2} + c(\kappa) \varepsilon^2\ ,
\end{align*}
hence
\begin{equation}
\left| \langle \partial_x V , \Gamma^2 \rangle_{H^{s-2},H^{2-s}} \right| \le c(\kappa) \varepsilon^{-1-2\nu} Q^{(2\nu+3)/2} + c(\kappa) \varepsilon^2\ . \label{x3}
\end{equation}
Also, by \eqref{x2}, \eqref{ch}, the definition \eqref{ce} of $Q$, and the Cauchy-Schwarz inequality,
\begin{align*}
\left| \int_\Omega V \left( \partial_x\partial_\eta \Phi - \eta V \partial_\eta^2 \Phi \right) \partial_\eta \Phi\ \rd(x,\eta) \right| & \le \|\partial_x v\|_\infty  \left\| \partial_x\partial_\eta \Phi - \eta V \partial_\eta^2 \Phi \right\|_{L_2(\Omega)} \left\| \frac{\partial_\eta \Phi}{1+v} \right\|_{L_2(\Omega)} \\
& \le c(\kappa) Q\ .
\end{align*} 
Recalling \eqref{x3} we end up with the following estimate for $\mathcal{R}_1^*$:
\begin{equation}
|\mathcal{R}_1^*| \le c(\kappa) \left[ \varepsilon^{-1-2\nu} Q^{(2\nu+3)/2} + Q + \varepsilon^2 \right]\ . \label{x4}
\end{equation}
We next turn to $\mathcal{R}_2^*$ and first deduce from \eqref{x2} and \eqref{x2.5} (with $\sigma=(2\nu+1)/2)$ that
\begin{align*}
\left| \langle \partial_x V , \Gamma \rangle_{H^{s-2},H^{2-s}} \right| & \le \|\partial_x V \|_{H^{s-2}} \|\Gamma\|_{H^{2-s}} \le c(\kappa) \|\Gamma\|_{H^{\nu}} \\
& \le c(\kappa) \|\partial_\eta\Phi\|_{H^1(\Omega)}^{(2\nu+1)/2} \|\partial_\eta\Phi\|_{L_2(\Omega)}^{(1-2\nu)/2}\ .
\end{align*}
Arguing as in the proof of \eqref{x3} leads us to
\begin{equation}
\left| \langle \partial_x V , \Gamma \rangle_{H^{s-2},H^{2-s}} \right| \le c(\kappa) \varepsilon^{-2\nu} Q^{(2\nu+1)/2} + c(\kappa) \varepsilon\ . \label{x5}
\end{equation}
In addition, it follows from \eqref{x2}, the definition \eqref{ce} of $Q$, and the Cauchy-Schwarz inequality,
$$
\left| \int_\Omega \eta V^2 \partial_\eta^2 \Phi\ \rd(x,\eta) \right| \le \|V\|_\infty \|\partial_x v\|_\infty \left\| \frac{\partial_\eta^2 \Phi}{1+v} \right\|_{L_2(\Omega)} \le c(\kappa) Q\ ,
$$
which gives, together with \eqref{x5},
\begin{equation}
|\mathcal{R}_2^*| \le c(\kappa) \left[ \varepsilon^{-2\nu} Q^{(2\nu+1)/2} + Q + \varepsilon \right]\ . \label{x6}
\end{equation}
We now infer from \eqref{ce}, \eqref{x4}, and \eqref{x6} that
$$
Q^2 \le c(\kappa) \left[ \varepsilon^{1-2\nu} Q^{(2\nu+3)/2} +  \varepsilon^{2-2\nu} Q^{(2\nu+1)/2} + \varepsilon^2 Q + \varepsilon^3 + \varepsilon^4 \right]
$$
and further deduce from Young's inequality (using $2\nu+3<4$ and $\e\in (0,1)$) that
$$
Q^2 \le \frac{Q^2}{2} + c(\kappa) \left[ \varepsilon^{8(1-\nu)/(3-2\nu)} + \varepsilon^3 + \varepsilon^4 \right] \le \frac{Q^2}{2} + c(\kappa) \varepsilon^{8(1-\nu)/(3-2\nu)}\ .
$$
Recalling \eqref{ch} we have thus established that 
\begin{align}
\left\| \partial_x \Phi - \eta V \partial_\eta\Phi \right\|_{L_2(\Omega)} + \frac{1}{\varepsilon} \left\| \frac{\partial_\eta \Phi}{1+v} \right\|_{L_2(\Omega)} & \le c(\kappa) \ , \label{x7} \\
\left\| \partial_x\partial_\eta \Phi - \eta V \partial_\eta^2 \Phi \right\|_{L_2(\Omega)} + \frac{1}{\varepsilon} \left\| \frac{\partial_\eta^2 \Phi}{1+v} \right\|_{L_2(\Omega)} & \le c(\kappa) \varepsilon^{(1-2\nu)/(3-2\nu)} \ , \label{x8}
\end{align}
with $1-2\nu>0$. Several consequences can be derived from \eqref{x7}-\eqref{x8}. First, it readily follows from the boundary condition $\Phi(x,0)=0$ for all $x\in I$ that 
$$
\Phi(x,\eta) = \int_0^\eta \partial_\eta \Phi(x,\xi)\ d\xi\ , \qquad (x,\eta)\in\Omega\ ,
$$
which, together with \eqref{x7}, gives
\begin{equation}
\|\Phi\|_{L_2(\Omega)} \le \|\partial_\eta\Phi\|_{L_2(\Omega)} \le c(\kappa)\varepsilon\ . \label{x9}
\end{equation}
Another straightforward consequence of \eqref{x2}, \eqref{x7}, \eqref{x8}, and the continuity \eqref{x2.5} of the trace operator is 
\begin{equation}
\|\Gamma\|_{H^{1/2}} \le c \left\|\partial_\eta\Phi \right\|_{H^1(\Omega)} \le c(\kappa) \varepsilon^{(1-2\nu)/(3-2\nu)}\ . \label{x10}
\end{equation} 
Now, owing to the relationship \eqref{c0} between $\Phi$ and $\psi$, we realize that, for $(x,z)\in\Omega(v)$ and $\eta=(1+z)/(1+v(x))$,
\begin{align*}
\partial_z \psi(x,z) - \frac{1}{1+v(x)} & = \frac{\partial_\eta\Phi(x,\eta)}{1+v(x)} \ , \\
\partial_x \partial_z \psi(x,z) + \frac{\partial_x v(x)}{(1+v(x))^2} & = \frac{\partial_x\partial_\eta\Phi(x,\eta) - \eta V(x) \partial_\eta^2 \Phi(x,\eta) - V(x) \partial_\eta \Phi(x,\eta)}{1+v(x)} \ , \\
\partial_z^2 \psi(x,z) & = \frac{\partial_\eta^2 \Phi(x,\eta)}{(1+v(x))^2}\ .
\end{align*}
Combining these formulas with \eqref{x7}-\eqref{x9} leads to \eqref{errest1}, the first part of \eqref{errest2}, and \eqref{errest3}. As for the second part of \eqref{errest2} we first recall that, thanks to \eqref{rP}, there holds $\psi(x,v(x))=1$ for $x\in I$ from which we deduce that $\partial_x\psi(x,v(x))=-\partial_x v(x) \partial_z \psi(x,v(x))$, $x\in I$. Thus
$$
g_\e(v) = \left( \frac{1}{(1+v)^2} + \varepsilon^2 V^2 \right)\ \left( 1 + \Gamma \right)^2 = \frac{1}{(1+v)^2} + \frac{\Gamma(2+\Gamma)}{(1+v)^2}  + \varepsilon^2 V^2 (1+\Gamma)^2\ .
$$
Given $\sigma\in [0,1/2)$, continuity of pointwise multiplication (see \cite[Theorem~4.1 \& Remark~4.2(d)]{Am91})
$$
H^{1/2}(I)\cdot H^{1/2}(I) \longrightarrow H^{(1+2\sigma)/4}(I)\,, \qquad H^{s-1}(I)\cdot H^{(1+2\sigma)/4}(I) \longrightarrow H^\sigma(I)
$$
and \eqref{x2} entail that
\begin{align*}
\left\| g_\e(v) - \frac{1}{(1+v)^2} \right\|_{H^\sigma} & \le \left\| \frac{1}{(1+v)^2} \right\|_{H^{s-1}} \|\Gamma(2+\Gamma)\|_{H^{(1+2\sigma)/4}} + \varepsilon^2 \|V^2\|_{H^{s-1}} \|(1+\Gamma)^2\|_{H^{(1+2\sigma)/4}} \\
& \le c(\kappa) \|\Gamma\|_{H^{1/2}} \|2+\Gamma\|_{H^{1/2}} + c(\kappa) \varepsilon^2\ \|1+\Gamma\|_{H^{1/2}}^2 \ .
\end{align*}
Therefore, thanks to \eqref{x10}, 
$$
\left\| g_\e(v) - \frac{1}{(1+v)^2} \right\|_{H^\sigma} \le c(\kappa) \varepsilon^{(1-2\nu)/(3-2\nu)}\ ,
$$
which completes the proof of \eqref{errest2}.
\end{proof}

\begin{remark}\label{rem1}
As already mentioned, estimates similar to \eqref{x7} and \eqref{x8} have already been derived in \cite{ELW14, LW13} when $v$ is more regular, namely $v\in S_q^2(\kappa)$ for some $q>2$ and $\kappa$. The estimates obtained therein are better in the sense that they involve higher powers of $\varepsilon$. A rough explanation for this discrepancy is that we use several times Green's formula in the proof of Proposition~\ref{prop2} to handle less regular functions $v$. This procedure somewhat mixes the $x$-derivative and $\eta$-derivative which do not decay in the same way with respect to $\varepsilon$ and results in the weaker estimates \eqref{x7} and \eqref{x8}.
\end{remark}

\section{Small aspect ratio limit: the stationary case}\label{sec3}

Fix $\varrho>2$. Starting from the outcome of Proposition~\ref{prop1} the first step of the proof of Theorem~\ref{thm2} is to establish bounds on $(\lambda_{\varrho,\varepsilon}, u_{\varrho,\varepsilon}, \psi_{\varrho,\varepsilon})$ which do not depend on $\varepsilon$ small enough.

\begin{lemma}\label{lem3.1}
There are $C_1>0$, $\kappa_0\in (0,1)$, and $\varepsilon_0>0$ such that, for $\varepsilon\in (0,\varepsilon_0)$,
\begin{align}
& -1+\kappa_0 \le u_{\varrho,\varepsilon}(x) \le 0\ , \qquad x\in [-1,1]\ , \label{s3} \\
& \lambda_{\varrho,\varepsilon} + \|u_{\varrho,\varepsilon}\|_{H^2} \le \frac{1}{\kappa_0}\ , \label{s4} \\
& 0 \le \varrho - \int_{-1}^1 \frac{\rd x}{1+u_{\varrho,\varepsilon}(x)} \le C_1 \varepsilon^2\ . \label{s5}
\end{align}
\end{lemma}

\begin{proof}
Setting 
$$
\mathcal{K}^1 := \left\{ v \in H^1_D(I)\ :\ -1 < v \le 0 \;\text{ in }\; I \right\}\ ,
$$
we recall that $\mathcal{E}_{e,\e}\in C(\mathcal{K}^1)$ by \cite[Proposition~2.7]{LW_COCV} and satisfies
\begin{equation}
\int_{-1}^1 \frac{\rd x}{1+v} \le \mathcal{E}_{e,\e}(v) \le \int_{-1}^1 \left( 1 + \e^2 |\partial_x v|^2 \right)\ \frac{\rd x}{1+v}\ , \quad v\in\mathcal{K}^1\ , \label{s2} 
\end{equation}
see \cite[Lemma~2.8]{LW_COCV}. Introducing $\phi(x) := - (1-x^2)^2$, $x\in I$, it readily follows from the continuity of $\mathcal{E}_{e,\e}$ and \eqref{s2} that $\theta\mapsto \mathcal{E}_{e,\e}(\theta\phi)$ continuously maps $(0,1)$ in $(2,\infty)$. Consequently, there is $\theta_{\varrho,\varepsilon}\in (0,1)$ such that $\mathcal{E}_{e,\e}(\theta_{\varrho,\varepsilon}\phi)=\varrho$ and thus $\theta_{\varrho,\varepsilon}\phi$ belongs to the set $\mathcal{A}_{\varrho,\varepsilon}$ introduced in \eqref{s1b}. The variational characterization \eqref{s1a} of $u_{\varrho,\varepsilon}$ then entails that
\begin{equation}
\mathcal{E}_m(u_{\varrho,\varepsilon}) \le \mathcal{E}_m(\theta_{\varrho,\varepsilon}\phi) \le \mathcal{E}_m(\phi)\ . \label{s6}
\end{equation}
Furthermore it follows from \cite[Equation~(6.12)]{LW14a} that there are $\varepsilon_0>0$ and $\Lambda>0$ such that \eqref{eqpsi}-\eqref{bcu} has no stationary solution for $\lambda>\Lambda$ and $\varepsilon\in (0,\varepsilon_0)$. Consequently, $\lambda_{\varrho,\varepsilon}\in (0,\Lambda]$ for $\varepsilon\in (0,\varepsilon_0)$ which, together with \eqref{s6}, gives \eqref{s4}. 

Next, since $u_{\varrho,\varepsilon}\in\mathcal{A}_{\varrho,\varepsilon}$, we infer from Proposition~\ref{prop1}, \eqref{s4}, and \cite[Lemma~3.3]{LW_COCV} that
$$
0 = \max_{[-1,1]} u_{\varrho,\varepsilon} \ge \min_{[-1,1]} u_{\varrho,\varepsilon} \ge - 1 +  \frac{\kappa_0^2}{\varrho (2\kappa_0 + \varrho)^2}\ ,
$$
whence \eqref{s3} by making $\kappa_0$ smaller, if necessary. 

Finally, using again that $u_{\varrho,\varepsilon}\in\mathcal{A}_{\varrho,\varepsilon}$, we deduce from \eqref{s3}, \eqref{s4}, and \eqref{s2} that
$$
0 \le \varrho - \int_{-1}^1 \frac{\rd x}{1+u_{\varrho,\varepsilon}} \le \varepsilon^2 \int_{-1}^1 \frac{|\partial_x u_{\varrho,\varepsilon}|^2}{1+u_{\varrho,\varepsilon}}\ \rd x \le \frac{\e^2}{\kappa_0} \|\partial_x u_{\varrho,\varepsilon}\|_2^2 \le \frac{\e^2}{\kappa_0^3}\ ,
$$
which completes the proof.
\end{proof}

Thanks to the just derived bounds and the analysis performed in Section~\ref{sec2} we are in a position to prove Theorem~\ref{thm2}.

\begin{proof}[Proof of Theorem~\ref{thm2}]
Owing to the compact embedding of $H^2(I)$ in $H^s(I)$, $s\in [1,2)$, and in $C^1([-1,1])$, we infer from Lemma~\ref{lem3.1} that there are a sequence $(\varepsilon_k)_{k\ge 1}$ with $\varepsilon_k\to 0$, $\lambda_\varrho\in [0,1/\kappa_0]$, and $u_\varrho\in H_D^2(I)$ such that
\begin{equation}
u_{\varrho,\varepsilon_k} \rightharpoonup u_\varrho \;\;\text{ in }\;\; H^2(I) \label{s7}
\end{equation} 
and
\begin{equation}
\lim_{k\to\infty} \left\{ |\lambda_{\varrho,\varepsilon_k} - \lambda_\varrho | + \| u_{\varrho,\varepsilon_k} - u_\varrho\|_{H^s} + \| u_{\varrho,\varepsilon_k} - u_\varrho\|_{C^1([-1,1])} \right\} = 0 \label{s8}
\end{equation}
for any $s\in [1,2)$. It readily follows from \eqref{s3}, \eqref{s5}, and \eqref{s8} that 
\begin{equation}
-1 + \kappa_0 \le u_\varrho(x)\le 0\ , \quad x\in I\,,\;\;\text{ and }\;\; \int_{-1}^1 \frac{\rd x}{1+u_\varrho(x)} = \varrho\ . \label{s9}
\end{equation}
Now, fix $s\in (3/2,2)$ and $\nu\in (2-s,1/2)$. According to \eqref{s3} and \eqref{s4}, $u_{\varrho,\varepsilon_k}$ belongs to $S_2^s(\kappa_0)$ and we infer from Proposition~\ref{prop2}, \eqref{s3}, and \eqref{s9} that, for $\sigma\in [0,1/2)$,
\begin{align*}
\left\| g_{\varepsilon_k}(u_{\varrho,\varepsilon_k}) - \frac{1}{(1+u_\varrho)^2} \right\|_{H^{\sigma}} & \le \left\| g_{\varepsilon_k}(u_{\varrho,\varepsilon_k}) - \frac{1}{(1+u_{\varrho,\varepsilon_k})^2} \right\|_{H^{\sigma}} + \left\| \frac{1}{(1+u_{\varrho,\varepsilon_k})^2} - \frac{1}{(1+u_\varrho)^2} \right\|_{H^{\sigma}} \\
& \le c(\kappa_0) \varepsilon_k^{(1-2\nu)/(3-2\nu)} + \left\| \frac{(2+ u_\varrho + u_{\varrho,\varepsilon_k}) (u_\varrho - u_{\varrho,\varepsilon_k})}{(1+u_{\varrho,\varepsilon_k})^2 (1+u_\varrho)^2} \right\|_{H^{\sigma}} \\
& \le c(\kappa_0) \left[ \varepsilon_k^{(1-2\nu)/(3-2\nu)} + \| u_\varrho - u_{\varrho,\varepsilon_k}\|_{H^1} \right]\ .
\end{align*}
We then deduce from \eqref{s8} that
\begin{equation}
\lim_{k\to\infty} \left\| g_{\varepsilon_k}(u_{\varrho,\varepsilon_k}) - \frac{1}{(1+u_\varrho)^2} \right\|_{H^{\sigma}} = 0 \label{s10}
\end{equation}
for all $\sigma\in [0,1/2)$. Thanks to \eqref{s8} and \eqref{s10} it is now straightforward to pass to the limit as $\varepsilon_k\to 0$ in the equation solved by $u_{\varrho,\varepsilon_k}$ and conclude that $u_\varrho$ is a weak solution in $H_D^2(I)$ to \eqref{vivaldi}. However, since the right-hand side of \eqref{vivaldi} belongs to $H^2_D(I)$, classical elliptic regularity results entail that  $u_\varrho$ belongs to $H_D^4(I)$ and is a classical solution to \eqref{vivaldi}. 

To check the minimizing property of $u_\varrho$, we consider $v\in \mathcal{A}_{\varrho,0}$ and observe that the function $\vartheta\mapsto \mathcal{E}_{\varrho,\varepsilon}(\vartheta v)$ continuously maps $(0,1]$ onto $(2,\mathcal{E}_{\varrho,\varepsilon}(v)]$, while $\varrho\le \mathcal{E}_{\varrho,\varepsilon}(v)$ for $\varepsilon>0$ according to \eqref{s2}. Consequently, there is $\vartheta_{\varrho,\varepsilon}\in (0,1]$ such that $\mathcal{E}_{\varrho,\varepsilon}(\vartheta_{\varrho,\varepsilon} v)=\varrho$, that is, $\vartheta_{\varrho,\varepsilon} v\in \mathcal{A}_{\varrho,\varepsilon}$. Recalling \eqref{s1a} gives
$$
\mathcal{E}_m(u_{\varrho,\varepsilon}) \le \mathcal{E}_m(\vartheta_{\varrho,\varepsilon} v) \le \mathcal{E}_m(v)\ .
$$
We then use the lower semicontinuity of $\mathcal{E}_m$ and \eqref{s7} to conclude that $\mathcal{E}_m(u_\varrho) \le \mathcal{E}_m(v)$. This inequality being valid for all $v\in\mathcal{A}_{\varrho,0}$, we have thus proved that $u_\varrho$ is a minimizer of $\mathcal{E}_m$ on $\mathcal{A}_{\varrho,0}$.

Finally the stated convergence properties of $\psi_{\varrho,\varepsilon_k}$ readily follow from Proposition~\ref{prop2} and \eqref{s8}.
\end{proof}

\section{Small aspect ratio limit: the evolutionary case}\label{sec4}

We next focus on the vanishing aspect ratio limit for the evolution problem \eqref{eqpsi}-\eqref{inu}. As pointed out in the introduction, the proof of the parabolic case $\gamma=0$ stated in Theorem~\ref{thm3} is similar to (actually, simpler than)  the hyperbolic case $\gamma>0$. We thus only treat the latter and may assume without loss of generality that $\gamma=1$.  As a first step we recall the well-posedness of \eqref{eqpsi}-\eqref{inu} which  is  established in \cite{LW14a}. Let $2\alpha\in (0,1/2)$ and consider the Hilbert space $\mathbb{H}_\alpha := H_D^{2+2\alpha}(I)\times H^{2\alpha}_D(I)$ and the (unbounded) linear operator $\mathbb{A}_\alpha$ on $\mathbb{H}_\alpha$ with domain $D(\mathbb{A}_\alpha) := H^{4+2\alpha}_D(I)\times H_D^{2+2\alpha}(I)$ defined by
$$
\mathbb{A}_\alpha \mathbf{w} := 
\begin{pmatrix}
0 & -w^1 \\
 & \\
\beta \partial_x^4 w^0 - \tau \partial_x^2 w^0 & w^1
\end{pmatrix} \ , \quad \mathbf{w}=(w^0,w^1)\in D(\mathbb{A}_\alpha)\ .
$$
According to \cite[Chapter~V]{Am95} and \cite{Ar04} the operator $\mathbb{A}_\alpha$ generates a strongly continuous group $(e^{-t\mathbb{A}_\alpha})_{t\in\mathbb{R}}$ on $\mathbb{H}_\alpha$ and \cite{HZ88} entails that the damping term provides an exponential decay, that is, there are $M_\alpha>0$ and $\omega>0$ such that 
\begin{equation}
\|e^{-t\mathbb{A}_\alpha}\|_{\mathcal{L}(\mathbb{H}_\alpha)}\le M_\alpha e^{-\omega t}\ ,\quad t\ge 0\ . \label{R0}
\end{equation}
Let $\e, \lambda>0$, $a\ge 0$, and fix $\kappa\in (0,1)$. Consider $\mathbf{u}^0 := (u^0,u^1)\in D(\mathbb{A}_\alpha)$ such that
\begin{equation}
u^0\in S_2^{2+2\alpha}(\kappa) \;\;\;\text{ and }\;\;\; \|u^1 \|_{H^{2\alpha}} < \frac{1}{\kappa}\ . \label{R01}
\end{equation}
By \cite[Corollary~3.4]{LW14a} there is a unique solution $(u_\varepsilon,\psi_\varepsilon)$ with $\psi_\varepsilon = \psi_{u_\varepsilon}$ to \eqref{eqpsi}-\eqref{inu} defined on the maximal interval of existence $[0,T_m^\varepsilon)$ satisfying
$$
\mathbf{u}_\varepsilon := (u_\varepsilon,\partial_t u_\varepsilon) \in C([0,T_m^\varepsilon),\mathbb{H}_\alpha) \ , \quad  \min_{x\in [-1,1]} u_\varepsilon(t,x) > -1 \,,\quad t\in [0,T_m^\varepsilon)\,,
$$
and
\begin{equation}
\mathbf{u}_\varepsilon(t) = e^{-t\mathbb{A}_\alpha} \mathbf{u}^0 +\int_0^t  e^{-(t-s)\mathbb{A}_\alpha}\, F_\varepsilon(\mathbf{u}_\varepsilon(s))\,\rd s\,, \quad t\in [0,T_m^\varepsilon)\ ,\label{vdk}
\end{equation}
where
\begin{equation}
F_\varepsilon(\mathbf{w}) := 
\begin{pmatrix} 
0\\
-\lambda\, g_\varepsilon(w^0)+ a\|\partial_x w^0\|_2^2\,\partial_x^2 w^0 
\end{pmatrix}\ , \qquad \mathbf{w} \in \mathbb{H}_\alpha\ . \label{rhs}
\end{equation}

After this preparation let us begin the study of the behavior as $\varepsilon\to 0$ by noticing that, setting $\kappa_0:=\kappa/2 \in (0,1/2)$,  the continuity properties of $\mathbf{u}_\varepsilon$ and \eqref{R01} ensure that
\begin{equation}
T^\varepsilon := \sup{\left\{ t\in [0,T_m^\varepsilon)\, :\, u_\varepsilon(s)\in S_2^{2+\alpha}(\kappa_0)\text{ and }
\|\partial_t u_\varepsilon(s)\|_{H^{2\alpha}} < 1/\kappa_0\,\text{ for } s\in [0,t]
 \right\}} > 0\ . \label{pam2}
\end{equation}
Thanks to the continuity of the embeddings of $H^{2+2\alpha}(I)$ in $W_q^2(I)$ for some $q>2$ and in $W_\infty^1(I)$ (the latter with embedding constant denoted by $c_I$), the definition of $T^\varepsilon$ guarantees that there is a positive constant $K_1$ depending only on $\kappa$ and $\alpha$ such that, for all $\varepsilon\in (0,1)$,
\begin{align}
-1 + \kappa_0 \le u_\varepsilon(t,x) & \le \frac{c_I}{\kappa_0}\,, \qquad (t,x)\in [0,T^\varepsilon)\times [-1,1]\,, \label{z1} \\
\|u_\varepsilon(t)\|_{W_q^2(I)} + \|u_\varepsilon(t)\|_{W_\infty^1(I)} & \le K_1\,, \qquad t\in [0,T^\varepsilon)\,. \label{z2}
\end{align}
 
The next step of the proof is to show that $T^\varepsilon$ (and thus also $T_m^\varepsilon$) does not collapse to zero as $\varepsilon\to 0$, so that the solutions $(u_\varepsilon,\psi_\varepsilon)_{\varepsilon\in (0,1)}$ have a common interval of existence.

\begin{lemma}\label{le.pam1}
\begin{itemize}
\item[(i)] There is $T>0$ depending only on  $\lambda$, $\alpha$, $a$, $\kappa$, and $\|\mathbf{u}^0\|_{D(\mathbb{A}_\alpha)}$ such that $T^\varepsilon\ge T$ for all $\varepsilon\in (0,1)$.
\item[(ii)] There is $\delta>0$ depending only on $\alpha$ and $\kappa$ such that $T^\varepsilon=T_m^\varepsilon=\infty$ for all $\varepsilon\in (0,1)$ provided $(\lambda, a , \|\mathbb{A}_\alpha\mathbf{u}^0\|_{\mathbb{H}_\alpha}) \in (0,\delta)\times [0,\delta)^2$.
\end{itemize}
Furthermore, there is $K_2>0$ depending only on $\kappa$ and $\alpha$ such that 
\begin{equation}
\left\| g_\varepsilon(u_\varepsilon(t)) \right\|_{H^{2\alpha}} \le K_2\ , \qquad t\in [0,T^\varepsilon)\ . \label{z3}
\end{equation}
\end{lemma}

\begin{proof}[{\bf Proof}]
Let $\varepsilon\in (0,1)$ and $t\in [0,T^\varepsilon)$. Since $2\alpha\in (0,1/2)$ and $u_\varepsilon(t)\in S_2^{2+2\alpha}(\kappa_0) \subset S_2^{(3+4\alpha)/2}(\kappa_0)$ we infer from Proposition~\ref{prop2} (with $s=(3+4\alpha)/2$ and $\sigma=2\alpha$), \eqref{z1}, and \eqref{z2} that
\begin{align}
\left\| g_\varepsilon(u_\varepsilon(t)) \right\|_{H^{2\alpha}} & \le \left\| g_\varepsilon(u_\varepsilon(t)) - \frac{1}{(1+u_\varepsilon(t))^2} \right\|_{H^{2\alpha}} + \left\| \frac{1}{(1+u_\varepsilon(t))^2} \right\|_{H^{2\alpha}} \nonumber \\
& \le C_{sg} + C \left\| \frac{1}{(1+u_\varepsilon(t))^2} \right\|_{W_\infty^1} \le K_3 \label{z4}
\end{align}
for some positive constant $K_3$ depending only $\kappa$ and $\alpha$, hence
\begin{equation}
\left\| F_\varepsilon(\mathbf{u}_\varepsilon(t)) \right\|_{\mathbb{H}_\alpha} \le (\lambda+a)\, K_3\ , \qquad t\in [0,T^\varepsilon)\ , \label{pam3}
\end{equation}
with a possibly larger constant $K_3$, but still depending only on $\alpha$ and $\kappa$.
Recalling that
$$
e^{-t\mathbb{A}_\alpha} \mathbf{u}^0 - \mathbf{u}^0 = -\int_0^t e^{-s\mathbb{A}_\alpha} \mathbb{A}_\alpha\mathbf{u}^0\,\rd s\ ,\quad t\ge 0\ ,
$$
it follows from \eqref{R0}, \eqref{vdk}, and \eqref{pam3} that, for $t\in [0,T^\varepsilon)$,
\begin{align}
\left\| \mathbf{u}_\varepsilon(t) - \mathbf{u}^0 \right\|_{\mathbb{H}_\alpha} & \le \left\| e^{-t\mathbb{A}_\alpha} \mathbf{u}^0 - \mathbf{u}^0 \right\|_{\mathbb{H}_\alpha} + \int_0^t \left\| e^{-(t-s)\mathbb{A}_\alpha} F_\varepsilon(\mathbf{u}_\varepsilon(s) ) \right\|_{\mathbb{H}_\alpha} \rd s \nonumber \\
& \le \frac{M_\alpha}{\omega}\,\left(1-e^{-\omega t}\right)\,\left( \|\mathbb{A}_\alpha\mathbf{u}^0\|_{\mathbb{H}_\alpha} + (\lambda+a) K_3 \right) \ . \label{R1}
\end{align}
Combining this estimate with \eqref{R01} further gives
\begin{align}
u_\varepsilon (t,x) & = u^0(x) + u_\varepsilon(t,x) -u^0(x) \ge -1+\kappa - \| u_\varepsilon(t) - u^0 \|_\infty \nonumber \\
& \ge -1 + \kappa - c_I\|\mathbf{u}_\varepsilon(t)  -\mathbf{u}^0\|_{\mathbb{H}_\alpha} \nonumber \\
& \ge -1+\kappa - \frac{M_\alpha c_I}{\omega}\, \left(1-e^{-\omega t}\right) \left( \|\mathbb{A}_\alpha\mathbf{u}^0\|_{\mathbb{H}_\alpha} +(\lambda+a) K_3 \right)\ , \label{R2}
\end{align}
for $(t,x)\in [0,T^\varepsilon) \times [-1,1]$.

On the one hand, if $T$ is chosen such that
$$
\frac{M_\alpha}{\omega}\, \left(1-e^{-\omega T}\right) \left( \|\mathbb{A}_\alpha\mathbf{u}^0\|_{\mathbb{H}_\alpha} +(\lambda+a) K_3 \right) < \min\left\{ \frac{2-\sqrt{2}}{\kappa} , \frac{\kappa}{2 c_I} \right\}\ ,
$$
then we deduce from \eqref{R1} and \eqref{R2} that, for $t\in [0,T)$, 
$$
\|\mathbf{u}_\varepsilon(t)\|_{\mathbb{H}_\alpha} = \sqrt{\|u_\varepsilon(t)\|_{H^{2+2\alpha}}^2 + \|\partial_t u_\varepsilon(t)\|_{H^{2\alpha}}^2} < \frac{1}{\kappa_0}\ ,
$$
while 
$$
u_\varepsilon(t,x) > - 1 + \kappa - \frac{\kappa}{2} = - 1+ \kappa_0\ , \qquad x\in [-1,1]\ .
$$
Consequently, $u_\varepsilon(t)\in S_2^{2+2\alpha}(\kappa_0)$ and $\|\partial_t u_\varepsilon(t)\|_{H^{2\alpha}} < 1/\kappa_0$ for all $t\in [0,T)$. We have thus shown that $T^\varepsilon\ge T$ and completed the proof of the first statement of Lemma~\ref{le.pam1}. 

On the other hand, let $\delta>0$ be such that
$$
\frac{M_\alpha}{\omega}\, \left( 1 +2 K_3 \right) \delta < \min\left\{ \frac{2-\sqrt{2}}{\kappa} , \frac{\kappa}{2 c_I} \right\}\ ,
$$
and assume that $\lambda\in (0,\delta)$, $a\in [0,\delta)$, and $\|\mathbb{A}_\alpha\mathbf{u}^0\|_{\mathbb{H}_\alpha} \in [0,\delta)$. Arguing as above we realize that, for all $t\in [0,T^\varepsilon)$, there hold
$$
u_\varepsilon(t)\in S_2^{2+2\alpha}(\kappa_0) \;\;\;\text{ and }\;\;\; \|\partial_t u_\varepsilon(t)\|_{H^{2\alpha}} < \frac{1}{\kappa_0}\ ,
$$
which entails that $T^\varepsilon=\infty$ and also that $T_m^\varepsilon=\infty$ as claimed in the second statement of Lemma~\ref{le.pam1}. Recalling \eqref{z4} completes the proof of Lemma~\ref{le.pam1}.
\end{proof}

\begin{proof}[{\bf Proof of Theorem~\ref{Bq}}]
According to Lemma~\ref{le.pam1}
$$
T^0 := \inf_{\varepsilon\in (0,1)} T^\varepsilon > 0\ .
$$
Fix $T\in (0,T^0)$. Recalling the definition \eqref{pam2} of $T^\varepsilon$, we realize that 
\begin{equation}
\text{ the family }\; (u_\varepsilon)_{\varepsilon \in (0,1)} \;\text{ is bounded in }\; C^1([0,T],H^{2\alpha}(I))\;\text{ and in }\; C([0,T],H^{2+2\alpha}(I))\ . \label{z5}
\end{equation} 
Also, since 
$$
\partial_t^2 u_\varepsilon = - \partial_t u_\varepsilon - \beta \partial_x^4 u_\varepsilon + \left( \tau + a \|  \partial_x u_\varepsilon\|_2^2 \right) \partial_x^2 u_\varepsilon - \lambda g_\varepsilon(u_\varepsilon) \;\;\text{ a.e. in }\;\; (0,T)\times I
$$
according to \eqref{equ} and \cite[Corollary~3.4]{LW14a}, we infer from \eqref{z3} and \eqref{z5} that 
\begin{equation}
\text{ the family }\; (\partial_t u_\varepsilon)_{\varepsilon \in (0,1)} \;\text{ is bounded in }\; C^1([0,T],H^{2\alpha-2}(I))\;\text{ and in }\; C([0,T],H^{2\alpha}(I))\ . \label{z6}
\end{equation}
Thus, given $2\alpha'\in (0,2\alpha)$, it follows from the compactness of the embeddings of $H^{2+2\alpha}(I)$ in $H^{2+2\alpha'}(I)$ and in $W_\infty^1(I)$, that of $H^{2\alpha}(I)$ in $H^{2\alpha'}(I)$, and the Arzel\`a-Ascoli theorem that there are a function 
$$
u_0\in C([0,T],H^{2+2\alpha'}(I)) \cap C^1([0,T],H^{2\alpha'}(I))
$$ 
and a sequence $(\e_k)_{k\ge 1}$ of positive real numbers with $\varepsilon_k \to 0$ such that 
\begin{equation}
\begin{split}
& \lim_{k\to\infty} \left\{ \sup_{t\in [0,T]} \| u_{\varepsilon_k}(t)-u_0(t)\|_{H^{2+2\alpha'}} + \sup_{t\in [0,T]} \| u _{\varepsilon_k}(t) - u_0(t) \|_{W_\infty^1} \right\}  = 0\ , \\
& \lim_{k\to\infty} \sup_{t\in [0,T]} \| \partial_t u_{\varepsilon_k}(t)-\partial_t u_0(t)\|_{H^{2\alpha'}} = 0\ .
\end{split} \label{z11}
\end{equation}
In particular,
\begin{equation}
\lim_{k\to\infty} \sup_{t\in [0,T]} \| \mathbf{u}_{\varepsilon_k}(t) - \mathbf{u}_0(t)\|_{\mathbb{H}_{\alpha'}} = 0 \;\;\text{ with }\;\; \mathbf{u}_0:=(u_0,\partial_t u_0)\ . \label{z12}
\end{equation}
A first consequence of \eqref{z1} and \eqref{z11} is 
\begin{equation}
-1+\kappa_0\le u_0(t,x)\le \frac{c_I}{\kappa_0}\ , \qquad (t,x)\in [0,T]\times [-1,1]\ . \label{R3}
\end{equation}
It also readily follows from Proposition~\ref{prop2} (with $s=2-\alpha$, $\nu=2\alpha$, and $\sigma=2\alpha'$), \eqref{z1}, \eqref{z11}, and \eqref{R3} that 
\begin{align*}
\sup_{t\in [0,T]} \left\| g_{\varepsilon_k}(u_{\varepsilon_k}(t)) - \frac{1}{(1+u_0(t))^2} \right\|_{H^{2\alpha'}} & \le C_{sg}\ \varepsilon_k^{(1-4\alpha)/(3-4\alpha)} \\
& \quad + C \sup_{t\in [0,T]} \left\| \frac{1}{(1+u_{\varepsilon_k}(t))^2} - \frac{1}{(1+u_0(t))^2} \right\|_{W_\infty^1}
\end{align*}
with right-hand side converging to zero as $k\rightarrow \infty$, hence
$$
F_{\varepsilon_k}(\mathbf{u}_{\varepsilon_k}) \longrightarrow F_0(\mathbf{u}_0) := 
\begin{pmatrix} 
0 \\
\displaystyle{- \frac{\lambda}{(1+u_0)^2} + a \|\partial_x u_0\|_2^2\, \partial_x^2 u_0 }
\end{pmatrix} \quad\text{ in }\quad C([0,T],\mathbb{H}_{\alpha'})\,.
$$
We are then in a position to pass to the limit as $\varepsilon_k\to 0$ in \eqref{vdk} and deduce from \eqref{z12} and the above convergence that
\begin{equation}
\mathbf{u}_0(t) = e^{-t\mathbb{A}_{\alpha'}} \mathbf{u}^0 +\int_0^t  e^{-(t-s)\mathbb{A}_{\alpha'}}\, F_0(\mathbf{u}_0(s))\,\rd s \ , \quad t\in [0,T]\ . \label{u}
\end{equation} 
In other words, $\mathbf{u}_0$ is a mild solution in $\mathbb{H}_{\alpha'}$ to
\begin{equation}\label{S}
\partial_t \mathbf{u}_0 +\mathbb{A}_{\alpha'}\mathbf{u}_0 =F_0( \mathbf{u}_0)\,,\quad t\in (0,T]\,,\qquad \mathbf{u}_0(0)=\mathbf{u}^0\,.
\end{equation}
Furthermore, $\mathbb{H}_{\alpha'}$ is reflexive and $F_0$ is a locally Lipschitz continuous map from $S_2^{2+2\alpha'}(\kappa_0)\times H^{2\alpha'}(I)$ into itself. Thus, since $\mathbf{u}^0\in D(\mathbb{A}_{\alpha'})$, we infer from \cite[Theorem~6.1.6]{Pa83} that $\mathbf{u}_0$ is actually a strong solution to \eqref{S}, that is, $\mathbf{u}_0\in L_1(0,T,D(\mathbb{A}_{\alpha'}))$ is differentiable a.e. with $\partial_t \mathbf{u}_0 = (\partial_t u_0, \partial_t^2 u_0)\in L_1(0,T,\mathbb{H}_{\alpha'})$. Therefore, $\mathbf{u}_0$ is a strong solution to \eqref{sg4.1}-\eqref{sg4.3}. Finally, the stated convergence of the  sequence $(\psi_{\varepsilon_k})_{k}$ readily follows from \eqref{errest1} and \eqref{z12}. 

In fact, we have so far proved Theorem~\ref{thm2} only for a sequence $(\varepsilon_k)_k$. However, the strong solution $u_0$ to \eqref{sg4.1}-\eqref{sg4.3} is unique, see \cite{LW14b} for a proof when $a=0$ which extends to the case $a>0$. This ensures that the whole family $(u_\varepsilon)_{\varepsilon\in (0,1)}$ converges, thereby completing the proof of Theorem~\ref{thm2}.
\end{proof}

\section*{Acknowledgments}

Part of this work was done while PhL enjoyed the hospitality of the Institut f\"ur Angewandte Mathematik, Leibniz Universit\"at Hannover.

\bibliographystyle{siam} 
\bibliography{MEMS}

\begin{thebibliography}{10}

\bibitem{Am91}
{\sc H.~Amann}, {\em Multiplication in {S}obolev and {B}esov spaces}, in
  Nonlinear analysis, Sc. Norm. Super. di Pisa Quaderni, Scuola Norm. Sup.,
  Pisa, 1991, pp.~27--50.

\bibitem{Am95}
\leavevmode\vrule height 2pt depth -1.6pt width 23pt, {\em Linear and
  quasilinear parabolic problems. {V}ol. {I}}, vol.~89 of Monographs in
  Mathematics, Birkh\"auser Boston, Inc., Boston, MA, 1995.
\newblock Abstract linear theory.

\bibitem{Ar04}
{\sc W.~Arendt}, {\em Semigroups and evolution equations: functional calculus,
  regularity and kernel estimates}, in Evolutionary equations. {V}ol. {I},
  Handb. Differ. Equ., North-Holland, Amsterdam, 2004, pp.~1--85.

\bibitem{COG09}
{\sc D.~Cassani, J.~M. do~{\'O}, and N.~Ghoussoub}, {\em On a fourth order
  elliptic problem with a singular nonlinearity}, Adv. Nonlinear Stud., 9
  (2009), pp.~177--197.

\bibitem{CFT14}
{\sc D.~Cassani, L.~Fattorusso, and A.~Tarsia}, {\em Nonlocal dynamic problems
  with singular nonlinearities and applications to {MEMS}}, in Analysis and
  topology in nonlinear differential equations, vol.~85 of Progr. Nonlinear
  Differential Equations Appl., Birkh\"auser/Springer, Cham, 2014,
  pp.~187--206.

\bibitem{CEGM10}
{\sc C.~Cowan, P.~Esposito, N.~Ghoussoub, and A.~Moradifam}, {\em The critical
  dimension for a fourth order elliptic problem with singular nonlinearity},
  Arch. Ration. Mech. Anal., 198 (2010), pp.~763--787.

\bibitem{DFG10}
{\sc J.~D{\'a}vila, I.~Flores, and I.~Guerra}, {\em Multiplicity of solutions
  for a fourth order equation with power-type nonlinearity}, Math. Ann., 348
  (2010), pp.~143--193.

\bibitem{ELW14}
{\sc J.~Escher, {\relax Ph}.~Lauren{\c{c}}ot, and {\relax Ch}.~Walker}, {\em A
  parabolic free boundary problem modeling electrostatic {MEMS}}, Arch. Ration.
  Mech. Anal., 211 (2014), pp.~389--417.

\bibitem{EGG10}
{\sc P.~Esposito, N.~Ghoussoub, and Y.~Guo}, {\em Mathematical analysis of
  partial differential equations modeling electrostatic {MEMS}}, vol.~20 of
  Courant Lecture Notes in Mathematics, Courant Institute of Mathematical
  Sciences, New York; American Mathematical Society, Providence, RI, 2010.

\bibitem{Flo14}
{\sc G.~Flores}, {\em Dynamics of a damped wave equation arising from {MEMS}},
  SIAM J. Appl. Math., 74 (2014), pp.~1025--1035.

\bibitem{Guo10}
{\sc Y.~Guo}, {\em Dynamical solutions of singular wave equations modeling
  electrostatic {MEMS}}, SIAM J. Appl. Dyn. Syst., 9 (2010), pp.~1135--1163.

\bibitem{GLY14}
{\sc Z.~Guo, B.~Lai, and D.~Ye}, {\em Revisiting the biharmonic equation
  modelling electrostatic actuation in lower dimensions}, Proc. Amer. Math.
  Soc., 142 (2014), pp.~2027--2034.

\bibitem{HZ88}
{\sc A.~Haraux and E.~Zuazua}, {\em Decay estimates for some semilinear damped
  hyperbolic problems}, Arch. Rational Mech. Anal., 100 (1988), pp.~191--206.

\bibitem{KLNT11}
{\sc N.~I. Kavallaris, A.~A. Lacey, C.~V. Nikolopoulos, and D.~E. Tzanetis},
  {\em A hyperbolic non-local problem modelling {MEMS} technology}, Rocky
  Mountain J. Math., 41 (2011), pp.~505--534.

\bibitem{KLNT15}
{\sc N.~I. Kavallaris, A.~A. Lacey, C.~V. Nikolopoulos, and D.~E. Tzanetis},
  {\em On the quenching behaviour of a semilinear wave equation modelling
  {MEMS} technology}, Discrete Contin. Dyn. Syst., 35 (2015), pp.~1009--1037.

\bibitem{LW_COCV}
{\sc {\relax Ph}.~Lauren{\c{c}}ot and {\relax Ch}.~Walker}, {\em A variational
  approach to a stationary free boundary problem modeling {MEMS}}, ESAIM
  Control Optim. Calc. Var. , to appear ({\tt arXiv:}{\tt 1409.2812v1
  [math.AP]}).

\bibitem{LW13}
\leavevmode\vrule height 2pt depth -1.6pt width 23pt, {\em A stationary free
  boundary problem modeling electrostatic {MEMS}}, Arch. Ration. Mech. Anal.,
  207 (2013), pp.~139--158.

\bibitem{LW14b}
\leavevmode\vrule height 2pt depth -1.6pt width 23pt, {\em A fourth-order model
  for {MEMS} with clamped boundary conditions}, Proc. Lond. Math. Soc. (3), 109
  (2014), pp.~1435--1464.

\bibitem{LW14a}
\leavevmode\vrule height 2pt depth -1.6pt width 23pt, {\em A free boundary
  problem modeling electrostatic {MEMS}: {I}. {L}inear bending effects}, Math.
  Ann., 360 (2014), pp.~307--349.

\bibitem{LLZ14}
{\sc C.~Liang, J.~Li, and K.~Zhang}, {\em On a hyperbolic equation arising in
  electrostatic {MEMS}}, J. Differential Equations, 256 (2014), pp.~503--530.

\bibitem{LiZ14}
{\sc C.~Liang and K.~Zhang}, {\em Global solution of the initial boundary value
  problem to a hyperbolic nonlocal {MEMS} equation}, Comput. Math. Appl., 67
  (2014), pp.~549--554.

\bibitem{LiZ15}
\leavevmode\vrule height 2pt depth -1.6pt width 23pt, {\em Asymptotic stability
  and quenching behavior of a hyperbolic nonlocal {MEMS} equation}, Commun.
  Math. Sci., 13 (2015), pp.~355--368.

\bibitem{LiY07}
{\sc F.~Lin and Y.~Yang}, {\em Nonlinear non-local elliptic equation modelling
  electrostatic actuation}, Proc. R. Soc. Lond. Ser. A Math. Phys. Eng. Sci.,
  463 (2007), pp.~1323--1337.

\bibitem{LiL12}
{\sc A.~E. Lindsay and J.~Lega}, {\em Multiple quenching solutions of a fourth
  order parabolic {PDE} with a singular nonlinearity modeling a {MEMS}
  capacitor}, SIAM J. Appl. Math., 72 (2012), pp.~935--958.

\bibitem{LLS13}
{\sc A.~E. Lindsay, J.~Lega, and F.~J. Sayas}, {\em The quenching set of a
  {MEMS} capacitor in two-dimensional geometries}, J. Nonlinear Sci., 23
  (2013), pp.~807--834.

\bibitem{LiW08}
{\sc A.~E. Lindsay and M.~J. Ward}, {\em Asymptotics of some nonlinear
  eigenvalue problems for a {MEMS} capacitor. {I}. {F}old point asymptotics},
  Methods Appl. Anal., 15 (2008), pp.~297--325.

\bibitem{LiW11}
\leavevmode\vrule height 2pt depth -1.6pt width 23pt, {\em Asymptotics of some
  nonlinear eigenvalue problems modelling a {MEMS} capacitor. {P}art {II}:
  multiple solutions and singular asymptotics}, European J. Appl. Math., 22
  (2011), pp.~83--123.

\bibitem{Pa83}
{\sc A.~Pazy}, {\em Semigroups of linear operators and applications to partial
  differential equations}, vol.~44 of Applied Mathematical Sciences,
  Springer-Verlag, New York, 1983.

\bibitem{PB03}
{\sc J.~A. Pelesko and D.~H. Bernstein}, {\em Modeling {MEMS} and {NEMS}},
  Chapman \& Hall/CRC, Boca Raton, FL, 2003.

\end{thebibliography}
\end{document}